\documentclass{amsart}

\usepackage{geometry}
\usepackage{amsmath,amssymb}
\numberwithin{equation}{section}
\newtheorem{thm}{Theorem}[section]
\newtheorem{prop}[thm]{Proposition}
\newtheorem{lem}[thm]{Lemma}

\newtheorem{rem}{Remark}[section]

\newtheorem{oss}{Remark}
\newtheorem*{step}{Step}

\def\tr{\text{tr}}
\def\eps{\varepsilon}
\def\d{\delta}
\newcommand{\R}{{\mathbb R}}
\newcommand{\Mi}{{\mathbf M}}
\begin{document}
\title[Viscosity methods for large deviations of multiscale stochastic processes]{Viscosity methods for large deviations estimates of multiscale stochastic processes}
\author{Daria Ghilli   }
 \address{Daria Ghilli, Karl-Franzens University of Graz, Universitatplatz 3, Austria}
\email{daria.ghilli@uni-graz.at}
 \thanks{%Part of this work was developped while the author was a PhD student at the Dept. of Mathematics of the Univ. of Padova. 
 	This work was partially supported by the ERC advanced grant $668998$ (OCLOC) under the EU's H$2020$ research programme. }
 
\begin{abstract}
We study singular perturbation problems for second order HJB equations in an unbounded setting. The main applications are large deviations estimates for the short maturity asymptotics of  stochastic systems affected by a stochastic volatility, where the volatility is modelled by a process evolving at a faster time scale and satisfying some condition implying ergodicity.  %We consider   fast processes which satisfy some conditions implying ergodicity. 
%This article is a continuation of the previous work \cite{BCG} where  the same kind of analysis was carried out under conditions implying compactness of the fast variable.   The main aim of this paper is to consider unbounded fast processes and replace the compactness by some conditions implying ergodicity.
\end{abstract}

\maketitle 
\section{Introduction}

We study the asymptotic behaviour as $\eps \to 0$ of stochastic   systems in the form
\begin{equation}\label{sistema02}
\left\{
\begin{array}{ll}
d X_t =   \eps  \phi(X_t, Y_t) dt + \sqrt{2\eps }\sigma(X_t, Y_t) dW_t \quad &X_{%t_
0
}=x%_0
 \in \R^n,\\
dY_t=\eps^{1-\alpha} b(Y_t) dt + \sqrt{2\eps^{1-\alpha}}\tau(Y_t) dW_t \quad &Y_{%t_
0
}=y%_0
 \in \R^m,
\end{array}
\right.\,
\end{equation}%
where $\eps>0$, $\alpha\geq2$, %$X_t \in \mathbb{R}^n, Y_t \in \mathbb{R}^m$, 
$W_t$ is a standard  $m$-dimensional Brownian motion, 
%the functions $\phi(x,y)$, $\sigma(x,y)$, $b(y)$,$\tau(y)$ are $\Z^m$-periodic with respect to the %second variable $y$, 
the matrix $\tau$ is non-degenerate.
%and the coefficients of the system satisfy suitable assumptions (we refer to the following section for the precise assumptions).
This is a model of system where  %some
the variables $Y_t$ % in this case)
  evolve at a much faster time scale $%\tau
 s=\frac{t}{\eps^{\alpha}}$ than the other variables $X_t$. The aim is to  study the small time behaviour of the system as %when
 $\eps$ goes to $0$, so time has been rescaled in \eqref{sistema02} as $ t \mapsto \eps t$. %A rather complete analysis is carried out in \cite{BCG}%In this case    we are interested  in the study of the limit 
 Motivated by the applications to large deviations that we want to give, we study the behaviour of the following logarithmic functional of the trajectories of  \eqref{sistema02}
$$
v^\eps(t,x,y) := \eps \log E\left[e^{%\eps^{-1}
h(X_{t})/\eps} | %X_0=x, Y_0=y,
 (X.,Y.)\, \, \mbox{satisfy \eqref{sistema02}}\right],
$$
where $ h$ is a bounded continuous function and we characterize  $v^\eps$  as the solution of the Cauchy problem with initial data  $v^\eps(0,x,y)=h(x)$ for a fully nonlinear parabolic equation in $n+m$ variables. %(see Proposition \ref{prop:propequazione1} where we recall this result).  

Our first aim is to prove that, under suitable assumptions, the functions $v^\eps(t,x,y)$  converge %as  $\eps\to 0$
  to a function $v(t,x)$
characterized as the solution of the Cauchy problem for a first order Hamilton-Jacobi equation in $n$ space dimensions
\begin{equation}\label{eqn:limitunb}
v_t-\bar H(x,Dv)=0 \;\text{ in } ]0,T[\times \R^n,\quad v(0,x)=h(x),
\end{equation}
 for a suitable effective Hamiltonian $\bar{H}$.
%This is a singular perturbation problem for nonlinear HJB. 
The existing techniques to treat this kind of problems have been developped so far mainly under  assumptions  implying some kind of compactness of the fast variable.  
We refer mainly to the methods of \cite{ABM}, stemming from Evans' perturbed test function method for homogenization \cite{E1} and its extensions to singular perturbations \cite{ab, BA, BA2}. 
A standard hypothesis is for example the periodicity of the coefficients of the stochastic system with respect $Y_t$, which in particular implies the periodicity in $y$ of the solutions $v^\eps$. 
In \cite{BCG}  the author toghether with M. Bardi and A. Cesaroni studied small time behaviour for the system defined above under this main assumption of periodicity. In \cite{BCG} a rather complete analysis is carried out,  also the case $1<\alpha <2$ is considered (which we do not treat in the present paper, see Remark \ref{whysubno}) and several representation formulas for the effective Hamiltonian are given. 
%On the other hand,  the methods for the convergence used in \cite{BCG} rely strongly on the periodicity assumption. 

Aim of this  paper is  to consider unbounded fast processes by replacing the compactness with some condition implying ergodicity, i.e that the  process $Y_t$ has a unique invariant distribution (the long-run distribution) and that in the long term it becomes independent of the initial distribution. A quite natural condition is the following
\begin{equation}\label{standcondintro}
 b(y)\cdot y \leq -B |y|^2,  \quad  \mbox{ if } |y|> R, \quad \mbox{for some } B>0, R>0,
 \end{equation}
which is reminiscent of other similar  conditions about recurrence of diffusion processes
in the whole space (see for example \cite{BCM}, \cite{PV1}, \cite{PV2}, \cite{PV3}).
The interest in the analysis of such kind of systems  is in part related to the financial applications we have in mind; in particular, the assumption of periodicity of \cite{BCG} appears as a technical restriction in order to model volatility in financial
markets, see the empirical data and the discussion presented in \cite{FPS} and the
references therein.

We study two different regimes depending on how fast the volatility oscillates relative to the horizon length, namely the supercritical case $\alpha>2$ and the critical case $\alpha =2$. 
%We do not treat the subcritical case $\alpha <2$ since in this case the 
We identify the effective Hamiltonians in both cases  through the resolution of  two different ergodic problems. 
For $\alpha>2$ the ergodic problem is finding,  for any $(\bar{x},\bar{p}) \in \R^n \times \R^n$, a  couple $ \lambda \in \R$ and $w$-viscosity solution of the following  uniformly elliptic linear equation
\begin{equation}\label{eqn:cellacriticounb}
\lambda- \tr(\tau \tau^T(y) D^2w(y)) -b(y) \cdot Dw(y)- |\sigma(\bar{x},y)^T\bar{p}|^2=0.
\end{equation}
Note that  $\lambda=\bar{H}(\bar{x},\bar{p})$ is the effective Hamiltonian and we call $w$ the corrector by analogy with the theory of homogenization. In order to prove the existence of the effective Hamiltonian and of the corrector we  approximate  the ergodic problem by the so-called approximate $\d$-ergodic problem, namely
$$
\d w_{\d}(y)- \tr(\tau \tau^T(y) D^2w_{\d}(y)) -b(y) \cdot Dw_{\d}(y)- |\sigma(\bar{x},y)^T\bar{p}|^2=0.
$$
 The main result which allows us to conclude the existence is a $\d$-uniform local Lipschitz bound for $w_\d$ (see  Lemma \ref{stimagradunbsopracrit}). 
 For the uniqueness of the effective Hamiltonian, we rely on the ergodicity of the process $Y_t$ (encoded by the assumption \eqref{standcondintro}) and on the results of  Bardi, Cesaroni, Manca \cite{BCM}, where the effective Hamiltonian is uniquely determined by the explicit formula
 $$H(\bar x, \bar p)= \int_{\R^m}
|\sigma(\bar x, y)^T \bar p|^2 d\mu(y),
$$
where $\mu$ is the invariant probability measure on $\R^m$ of the
stochastic process
$$dY_t = b(Y_t)dt +\sqrt{2}
\tau(Y_t)dW_t.
$$

 In addition, we  prove the existence  of the  corrector, which is not investigated  in \cite{BCM}.
 
For $\alpha=2$ the ergodic problem is finding,  for any $(\bar{x},\bar{p}) \in \R^n \times \R^n$, a  couple $ \lambda \in \R$ and $w$-viscosity solution of the following uniformly elliptic equation with quadratic nonlinearity in the gradient
\begin{equation}\label{eqn:cellacriticounb}
\lambda- \tr(\tau \tau^T(y) D^2w(y))-|\tau(y)^TDw(y)|^2 -(b(y)+ \tau(y)^T\sigma(\bar{x},y)^T\bar{p}) \cdot Dw(y)- |\sigma(\bar{x},y)^T\bar{p}|^2=0. 
\end{equation} 
For the existence of the effective Hamiltonian and the corrector, we proceed analogously to the supercritical case, in particular we rely on an analogous $\d$-uniform local Lipschitz bound for the solution of the approximate $\d$-ergodic problem (see Lemma \ref{stimagradunbcrit}). 
 We prove the uniqueness of the effective Hamiltonian relying  on the results  by Ichihara \cite{ichiprimo}, where ergodic problems for  Bellman equations (in the case of a nonlinear quadratic term) are solved. For a representation formula  we refer to \cite{KSsec}, where $\bar{H}$ is written  as the convex conjugate of  a suitable operator over a space of measures.
 
The main result is the  convergence of the functions $v^\eps$ to the solution of the limit problem \eqref{eqn:limitunb}. The main difficulties stem from the unboundedness of the fast variable, and the methods used in \cite{BCG} have to be modified since they rely strongly on the periodicity assumption.
Our techniques are based on the perturbed test function method of \cite{E1}, \cite{ABM}, with some relevant adaptations to the unbounded setting. We mainly rely on the ergodicity of the fast process through the use of a \textit{Liapounov function} (see  Section \ref{secassprel}, subsection \ref{liapfun}) into the perturbed test function.
Further difficulties in the proof of the  convergence come from the nonlinearity of the equation satisfied by the $v^\eps$. Our strategy relies essentially on  a global Lipschitz bound for the corrector, which we prove as a consequence of a  global $\d$-uniform Lipschitz bound for the solution of the approximate $\d$-ergodic problems (Proposition \ref{globunifwd}).  
Note that in order to prove the uniqueness of the limit Hamiltonian we  rely  on a local gradient estimate, whereas in the proof of the convergence we need a global bound.

In order to prove Proposition \ref{globunifwd} and then  conclude the convergence,  condition \eqref{standcondintro} is not sufficient and we have to streghten it by considering
\begin{equation}\label{mainass}
b(y)=b-y, \quad \tau(y)=\tau \quad \mbox{ for } |y|\geq R_1
\end{equation}
for some $R_1>0, $ where $b \in \mathbb{R}^m$ is a vector, and $\tau$ is bounded and uniformly non-degenerate. In particular, \eqref{mainass} is satisfied by the Ornstein-Uhlenbeck process.

A significant part of the paper is devoted to the proof of Proposition \ref{globunifwd}. This can be considered one of our main results since it is crucial  to prove the convergence and moreover it is a non standard result, at least to our knowledge, for the type of equations we consider, namely uniformly elliptic equations either with linear Hamiltonians in the gradient (supercritical case), or with superlinear quadratic Hamiltonians (critical case). 
 The proof is in some part inspired by a method due to Ishii
and Lions \cite{IshiiL} (see also \cite{US},\cite{BIMg} and the references therein), which essentially allows to take profit of the uniform ellipticity of the equation to control the
 Hamiltonian terms.  However, we remark that usually the Ishii-Lions method allows to achieve bounds which depend  on the $L^\infty$-norm of the solution (at least if we do not assume any periodicity), whereas our result is a global  estimate in all the space independent of such norm. The fundamental hypothesis which enables us to achieve our result  is the Ornstein-Uhlenbeck nature of the fast process at infinity encoded by  assumption \eqref{mainass}.  
 
We recall some results in the literature related to gradient bounds for similar kinds of equations. 
Gradient bounds for
superlinear-type Hamiltonians can be found in Lions \cite{Lg} and Barles \cite{Bg}, see also Lions and
Souganidis \cite{LSAHP}, Barles and Souganidis \cite{BS} and more, recently, Cardaliaguet and Silvestre \cite{CSg} for nonlinear degenerate parabolic equations.
However, we remark that in the previous works the bounds depend usually on the $L^\infty$-norm of the solution. In \cite{BS} some results independent of the $L^\infty$-norm of the solutions are established but in periodic environments. We recall also the result of \cite{CDLP} by Capuzzo-Dolcetta, Leoni, Porretta for  coercive superlinear Hamiltonians, where a uniform gradient bound is proved, but in some H\"older norm and only in bounded domains. 
%We refer  also to Barles \cite{Bg4}, Cardaliaguet \cite{Cg}. 
Recently, uniform Lipschitz bounds on the torus for analogous equations as ours (and more general) have been established by Ley and Duc Nguyen in \cite{LNg}. 

Following the approach of \cite{FFK} and \cite{BCG}, we derive   a large deviation principle for the process $X^\eps_t$, more precisely we prove that the measures associated to the process $X_t$ in \eqref{sistema02} satisfy such a principle with good rate function
$$
I(x;x_0,t):= \inf\left[\int_0^t \! \bar{L}\left(\xi(s),\dot{\xi}(s)\right)\, ds \ \Big|\ \xi\in AC(0,t), \ \xi(0)=x_0, \xi(t)=x\right],
$$
where $\bar{L}$ is the \emph{effective Lagrangian} associated to $\bar H$ via convex duality. In particular we get that
\begin{equation*}%\label{ldp}
%\lim_{\epsilon \rightarrow 0} \epsilon \log 
P(X^\eps_t \in B)=e^%{\frac{-\inf_{x \in B} I(x;x_0,t) + o(1)}{\eps}} 
{-\inf_{x \in B} \frac{I(x;x_0,t)}{\eps} + o(\frac 1{\eps})},\; \text{ as } \eps\to 0 
\end{equation*}
for any open set $B\subseteq\R^n$.
We also apply %application of 
this result to find %will be 
 %asymptotic
 estimates of option prices  near maturity and an asymptotic formula for the implied volatility. Since the proofs are  analogous to those of Theorem $7.1$, Corollary $8.1$ and $8.2$ of \cite{BCG}, we omit them. 
 For a detailed review of the results of this paper and of \cite{BCG}, we refer to \cite{phdth}.

Our first motivation for the study of systems of the form \eqref{sistema02} comes from financial
models with stochastic volatility, where the vector $X_t$ represents
the log-prices of $n$ assets (under a risk-neutral probability measure) and its volatility
$\sigma$ is affected by a process $Y_t$ driven by another Brownian motion (often negatively correlated). 
We adopt the approach of Fouque, Papanicolaou, and Sircar \cite{FPS}, where it is argued that the bursty behaviour
of volatility observed in financial markets can be described by introducing a faster
time scale for a mean-reverting process $Y_t$ (as in \eqref{sistema02}, where the process $Y_t$ evolves on the faster time scale $s=\frac{t}{\eps^\alpha}$)).
For more details on the financial model and for more references on large deviations literature, we refer to the introduction of \cite{BCG}. 

 We finally recall the paper \cite{FFK}, where Feng, Fouque, and Kumar study analogous problems for system of the form we consider, only for $\alpha=2$ and $\alpha=4$, %a similar problem to 
 in the one-dimensional case  $n=m=1$, assuming that $Y_t$ is the Ornstein-Uhlenbeck process and the coefficients in the equation for $X_t$ do not depend on $X_t$.  Their methods are based on the approach to large deviations developed in \cite{FK}. 
 %has been studied by an explicit computation of the moment generating function of the stock price and its asymptotic analysis. 
Comparing to \cite{FFK}, we remark that we consider more general fast processes satisfying \eqref{mainass}, we treat vector-valued processes with $\phi$ and $\sigma$ %under
depending on $X_t$ in a rather general way and we study all the range $\alpha\geq 2$. Also, our methods are  different, mostly from the theory of viscosity solutions for fully nonlinear PDEs %with %making systematic links t
and from the theory of homogenization and singular perturbations for such equations.

\subsection*{Organization of the paper}
%The paper is organised as follows. 
In  Section \ref{secassprel} we give the assumptions on the stochastic volatility model  and we recall some preliminaries.  In Sections
$3$ and $4$ we analyse the ergodic problem and the properties of the effective Hamiltonian in
the critical ($\alpha= 2$) and supercritical case ($\alpha > 2$), respectively.
Section $5$ is devoted to the proof of the Lipschitz bounds for the solution of the ergodic problems for each regimes. In Section $6$ we prove the comparison principle for the limit equation \eqref{eqn:limitunb} and finally in Section $7$ we prove the convergence result for each regime of the functions $v^\eps$
to the unique viscosity solution of the limit problem \eqref{eqn:limitunb} with $\bar{H}$ identified in
the previous section.
%There is a large literature on large deviations for diffusions with vanishing noise, and some results has been extended to the the two-scale systems with vanishing noise in the slow variable. We refer to the introduction of \cite{BCG} for more details and references.\\

\section{Assumptions and preliminaries}\label{secassprel}
\subsection{The stochastic volatility model}\label{assnoncomp}
We consider fast mean-reverting processes  of  the following type
\begin{equation}\label{sistema20}
\left\{
\begin{array}{ll}
d X_t = \phi(X_t , Y_t ) dt + \sqrt{2}\sigma(X_t , Y_t )  dW_t, \quad & X_0 =x \in \R^n\\
d Y_t = \eps^{-\alpha}b(Y_t) dt + \sqrt{ 2 \eps^{-\alpha}} \tau(Y_t ) dW_t,  \quad & Y_0 =y \in \R^m,
\end{array}
\right.\,
\end{equation}
where $\eps >0, \alpha \geq 2$,$\phi:\R^n \times \R^m \rightarrow \mathbb{R}^n, \sigma:\mathbb{R}^n \times \mathbb{R}^m \rightarrow \Mi^{n,m}$ are bounded functions, Lipschitz continuous  in $(x,y)$, $b:\R^m\rightarrow \R^m$ is Lipschitz continuous, $ \tau:\mathbb{R}^m\rightarrow \Mi^{m,m}$ is  bounded, Lipschitz continuous and uniformly non degenerate, i.e.  satisfies
for some $\theta>0$
\begin{equation} \label{unifnondeg}
 \xi ^T\tau(y) \tau(y)^T %\cdot
  \xi =|\tau^T(y) \xi |^2 >\theta|\xi|^2\quad \mbox{for every}\, \, y\in \mathbb{R} , \xi \in \mathbb{R}^m.
\end{equation}
This assumptions will hold throughtout the paper.

%\vspace{0.3cm}
We state now the basic assumptions on $b$ and $\sigma$ which will hold throughtout the paper. Note that the following assumptions are fundamental to the resolution of the ergodic problem and the identification of the limit Hamiltonian, but  are not sufficient in order to prove the convergence result whose proof is given in Section $7$. In the following subsection  we will streghten them approprately, as already announced in the introduction.

We assume the following condition on the fast process which ensures the ergodicity and, in particular, the existence of a Liapounov function. For further remarks, we refer to subsection \ref{liapfun}.
%where we prove the existence of a Lyapounov function (see in particular Remark \ref{ossliaperg}).
\begin{enumerate}
\item[(E)] \label{assulen}
There exist $B>0$ and $R$ such that
$$
 b(y)\cdot y \leq -B |y|^2,  \quad  \mbox{ if } |y|> R; 
 $$
\end{enumerate}

Moreover, in the supercritical case $\alpha>2$, we assume that for every $x \in \R^n$, the function $y\rightarrow \sigma(x,y)$ is uniformly non degenerate, that is
 \begin{enumerate}
\item[(S1)]
 for each $x \in \R^n$, there exists $\nu>0$ such that
%,k\in \{1,\cdots n\}, j\in \{1,\cdots m\}  $ 
\begin{equation*}
|\sigma(x,y)^T\xi|^2\geq \nu |\xi|^2, \quad  \mbox{ for all }y\in \R^n.
\end{equation*}
\end{enumerate}
%\vspace{0.5cm}
\begin{oss}\rm{
The previous assumption (S1) of uniform non-degeneracy of the volatility is due to technical issues arising in the proof of the local gradient bound for the solution of the ergodic problem for $\alpha>2$. We refer in particular to the proof of Lemma \ref{stimagradunbsopracrit}.
}
\end{oss}
%Notice that \eqref{tauipo} and \eqref{sigmaipo} implies in particular that $\tau$ and $\sigma$ are Lipschitz with respect to the $y$ variables (uniformly in $x$ for $\sigma$).\\
%Througout the chapter we will use the following notation. Given a matrix $A \in \Mi^{d,r} $, we denote
%\begin{equation}\label{linftynorm}
%||A||_{\infty}=\sup_{i,j}A_{ij} \quad i \in \{1,\cdots,d\}, j\in \{1,\cdots r\}.
%\end{equation}

In order to study small time behaviour of the system \eqref{sistema20}, we rescale time $t\to \eps t$, for $0<\eps \ll 1$, so that the typical maturity will be of order $\eps$. Denoting the rescaled process by $X^\eps_t, Y^\eps_t$ we get
\begin{equation}\label{sistema22}
\left\{
\begin{array}{ll}
d X_t = \eps\phi(X_t , Y_t ) dt + \sqrt{2\eps}\sigma(X_t , Y_t )  dW_t, \quad & X_0 =x \in \R^n\\
d Y_t = \eps^{1-\alpha}b(Y_t) dt + \sqrt{ 2 \eps^{1-\alpha}} \tau(Y_t ) dW_t,  \quad & Y_0 =y \in \R^m.
\end{array}
\right.\,
\end{equation}

\subsection{Further assumption on the stochastic systems}\label{furthercond}
Now we introduce our main assumption   on the fast process, on which we strongly rely in sections $5$ and $7$.  We assume that $b$ and $\tau$ satisfy   the following condition:
%\vspace{0.3cm}

\begin{enumerate}
\item[(U)] \label{assulen}
there exist $b\in \mathbb{R}^m, \tau \in \Mi^{m,m}$ and $R_1$ such that
$$
 b(y)=b-y,  \quad \tau(y)=\tau\quad  \mbox{ if } |y|> R_1.
 $$
\end{enumerate}

%Throughtout the chapter, we will denote compactly by (U) the set of assumptions 1), 2) above.
%\vspace{0.5cm}

%\vspace{0.5cm}
%Note that we do not assume any compactness of the fast variable $y$, which is replaced by some sort of ergodicity of the fast process $Y_t$, encoded by assumption 1) of (U). This means that  the $Y_t$ process has a unique invariant distribution (the long-run distribution) and that in the long term it becomes independent of the initial distribution. 
%\vspace{0.5cm}

\begin{oss}\label{remu}\rm{

%We refer also to \cite{ichiprimo}, where analogous conditions to (U), but even more general, see  for example \eqref{lyap} and \eqref{liapcond} in subsection \ref{lipafun}, are assumed to ensure the ergodicity. 
Note that assumption (U)  is stronger than condition (E).
%and the usual conditions implying ergodicity (see subsection \ref{liapfun} and in particular Remark \ref{ossliaperg} where we recall further conditions).
 The reason of such a stronger assumption are due to the fact that  the usual conditions implying ergodicity are not sufficient  in order to prove the global Lipschitz bound for the corrector (Proposition \ref{globunifw}), which is  a key result on which we rely strongly in the proof of the convergence.

For example, assumption (U) is satisfied by Ornstein-Uhlenbeck type processes, i.e. processes $Y_t$ as in \eqref{sistema20} such that 
$$
 b(y)=b-y, \quad \tau(y)=\tau  \quad  \mbox{ for any } y \in \mathbb{R}^m,
$$
for some $b\in \R^m$ and $\tau \in \Mi^{m,m}$ non-degenerate. The Ornstein-Uhlenbeck process  is a classical example of a Gaussian process that admits a stationary probability distribution and in particular is a mean-reverting process, namely there is a long-term value towards the process \,\textquotedblleft tends to revert\textquotedblright\,.
 %In other words, in dimension $m=1$, this means  that if the current value of the process is less than the (long-term) mean, the drift will be positive; if the current value of the process is greater than the (long-term) mean, the drift will be negative.  This gives the process the name "mean-reverting."}

%\vspace{0.5cm}

%\begin{oss}\rm{
%In the supercritical case $\alpha >2$, we have to assum
%}
%\end{oss}
%\vspace{0.2cm}

}
\end{oss}

%\vspace{0.5cm}
Moreover, in the critical case $\alpha=2$ we assume the following further condition on the volatility  $\sigma$:
%for all $i\in \{1,\cdots m\}, j\in \{1,\cdots m\}$ 
 \begin{enumerate}
\item[(S2)]
 for all $x\in \R^n$, there exists $g \, : \, \R^m\times\R^m\rightarrow \R^+$ such that
%,k\in \{1,\cdots n\}, j\in \{1,\cdots m\}  $ 
\begin{equation*}\label{sigmaipo}
%|\sigma_{kj}(x,y)-\sigma_{kj}(x,z)|\leq g(y,z)|y-z| \quad \mbox{for all } y,z\in \R^m.
||\sigma(x,y)-\sigma(x,z)||_\infty\leq g(y,z)|y-z| \quad \mbox{for all } y,z\in \R^m,
\end{equation*}
and   $\forall \eps>0$, there exists $R_\eps>0$ such that $g(y,z)\leq \eps$ as  $|z|,|y|\geq R_\eps$.
\end{enumerate}

From now on, for convenience of notation, we denote compactly by $(S)$ the set of assumptions $(S1)$ and $(S2)$ as follows
 \begin{enumerate}
\item[(S)]
When $\alpha>2$ $\sigma$ satisfies $(S1)$, when $\alpha=2$ $\sigma$ satisfies $(S2)$.
\end{enumerate}

\begin{rem} \rm{
We use (S2)  to prove the Lipschitz bound for the corrector in the critical case (Proposition \ref{globunifwd}). In particular, we need to assume (S2) to treat the correlation term  $\tau(y)\sigma^T(\bar{x}, y)\bar{p}\cdot Dw_\d$, which appears in the ergodic problem for $\alpha=2$. 
On the contrary, for $\alpha>2$ the correlation term do not appear in the ergodic problem (see \eqref{eqn:deltacell} in the following) and then  we do not need assumption (S2). 

Assumption (S2) says, roughly speaking,  that the Lipschitz constant of $\sigma(x,\cdot)$, considered as a function on $\R^m$ for $x\in \R^n$ fixed, vanishes at infinity. 
At least to our point of view,    (S2) seems not  restrictive in the context of financial models,  since it influences the behaviour of  $\sigma$ only at infinity, which in general is not "seen" in the financial applications we are interested in. 
 %We need such kind of conditions mainly for technical reasons linked to the control of the growth of the solutions of the cell problems. 
Examples of sufficient conditions for (S2) are
\begin{itemize}
\item[]
$$
 \lim_{|y|\to + \infty}g(y,z)=0 \quad \mbox{uniformly in } z,\\
 $$
 \item []
 $$
  \lim_{|z|\to + \infty}g(y,z)=0 \quad \mbox{uniformly in } y.
 $$
\end{itemize}
For example, the above conditions are satisfied by $\sigma(x,y)=\frac{1}{(1+|y|^2)^{\alpha}}$, for $\alpha>0$. Then in this case we have (S2) with $g(y,z)=\frac{C}{1+|y|+|z|}$.  Without loss of generality we suppose $n=1$ and $z\geq y\geq 0$. Then
$$
\sigma(y)-\sigma(z)=\frac{1}{(1+y^2)^{\alpha}}\left (1-\left(1+\frac{y^2-z^2}{1+z^2}\right)^{\alpha}\right).
$$
From the inequality $1-(1+x)^{\alpha}\leq -x$ for $-1\leq x\leq 0$, we get
$$
\sigma(y)-\sigma(z)\leq \frac{1}{(1+y^2)^{\alpha}}\frac{(z-y)(z+y)}{1+z^2}\leq \frac{2z}{1+z^2}(y-z).
$$
Since we assumed $z\geq y\geq 0$, we can find a constant $C$ independent of $y,z$ such that $\frac{2z}{1+z^2}\leq \frac{C}{1+z+y}$, concluding the proof. 
%\item[(ii)] 

%Note that in the previous example $o(y,z)\to 0$ as $|y|\to + \infty$ uniformly in $z$ (and the same as $|z|\to + \infty$ uniformly in $y$). Actually this is not necessary in assumption (S), for example we could have also $o(y,z)=\frac{|y|^2}{|z|}$, or equivalently $o(y,z)=\frac{|z|^2}{|y|}$.
}
\end{rem}

\subsection{The logarithmic transformation method and the HJB equation}
 We consider the following functional
\begin{equation}\label{v-eps}
v^\eps(t,x,y) := \eps \log E\left[e^{%\eps^{-1}
h(X_{t})/\eps} | %X_0=x, Y_0=y,
 (X.,Y.)\, \, \mbox{satisfy \eqref{sistema22}}\right],
\end{equation}
where $ h\in BC(\R^n)$ and $(X_s,Y_s)$ satisfies \eqref{sistema22}. Note that the logarithmic form of this payoff is motivated by the applications to large deviations that we want to give.

A standard result is that  $v^\eps$ can be characterized as  the unique continuous viscosity solution of the following parabolic problem.
 %For the proof we refer the reader to 
We refer  to  Da Lio and Ley in \cite{DLL} for a proof.
%\vspace{0.5cm}

\begin{prop}\label{prop:propequazione1} 
 Let  $\alpha \geq2$ and define 
\begin{eqnarray*} 
H^{\eps}(x,y, p, q, X,Y, Z)&:=&   |\sigma^T p|^2+ b\cdot q+ \tr(\tau \tau^T Y) +  \eps\left(\tr(\sigma\sigma^T X) + \phi \cdot p\right) %+
\\ &+& 2\eps^{\frac{\alpha}{2}-1} (\tau \sigma^T p) \cdot q   +2\eps^{\frac{1}{2}}\tr(\sigma\tau^T Z) + \eps^{\alpha-2}|\tau^T q|^2  .
\end{eqnarray*}
Then $v^\eps$  %as in \eqref{eqn:veps}
 is the unique bounded continuous viscosity solution of the %following
  Cauchy problem
\begin{equation}\label{eqn:equazioneeps11}
\begin{cases} 
 \partial_t v^\eps -H^\eps \left(x,y,D_x v^\eps, \frac{D_y v^\eps}{\eps^{\alpha-1}}, D^2_{xx} v^\eps, \frac{D^2_{yy} v^\eps}{\eps^{\alpha-1}}, \frac{D^2_{xy} v^\eps}{\eps^{\frac{\alpha-1}{2}}}\right)=0 & \, \, \mbox{in} \, \, [0,T] \times \mathbb{R}^n \times \mathbb{R}^m,\\
v^\eps(0,x,y)=h(x) & \mbox{ in }  \mathbb{R}^n\times \mathbb{R}^m. 
\end{cases}
\end{equation}
%where $h(x) \in BC(\mathbb{R}^n)$.  
\end{prop}
%\vspace{0.5cm}

\begin{oss}\label{whysubno}\rm{
We treat the range $\alpha \geq 2$ and  we do not deal with the  case $\alpha <2$. Indeed, for $\alpha <2$, the ergodic problem is finding (and characterizing it uniquely)  $\lambda \in \mathbb{R}$ and a $w$-viscosity solution of the following equation:
\begin{equation}\label{ergunb}
\lambda -  2(\tau(y) \sigma(\bar{x}, y)^T \bar{p}) \cdot D_y w(y) - |\tau(y)^T D_y w(y)|^2 -|\sigma(\bar{x},y)^T \bar{p}|^2=0, 
\end{equation}
which is not solvable in general. This is essentially due to the fact that  the  ergodicity of the fast process plays no role in  \eqref{ergunb}, since the  cost ($|\sigma^T\bar{p}|^2$) and the drift $2\tau\sigma^T\bar{p}$ are both bounded and the drift $b$ has disappeared. On the contrary, in the case  $\alpha \geq 2$, this role is played by the term $-b\cdot Dw$ where $b$ satisfies assumption (E). Finally we remark that in \cite{BCG}  the case $\alpha <2$ is solved thanks to the periodicity assumption.
}
\end{oss}
%\vspace{0.5cm}

\subsection{A Liapounov-like condition}\label{liapfun}
In this section we prove the existence of a Liapounov function for the following operator
$$
\mathcal{G}_{\bar{x},\bar{p}}(y,q,Y)=-(b(y)+2\tau(y)\sigma^T(\bar{x},y)\bar{p}) \cdot q -|\tau^T(y) q|^2 - \tr(\tau \tau^T(y) Y),
$$
  i.e. we prove that for each $(\bar{x},\bar{p})\in \mathbb{R}^n\times\R^n$ there exists a continuous function $\chi_{\bar{x},\bar{p}}:=\chi$, such that $\chi(y) \to + \infty \mbox{ as } |y| \to + \infty$ and if $\mathcal{G}[\chi]:=\mathcal{G}_{\bar{x},\bar{p}}(y,D\chi(y),D^2\chi(y))$ then 
\begin{equation}\label{lyap}
 \mathcal{G}[\chi]\to +\infty \mbox{ as } |y| \to +\infty \mbox{ in the viscosity sense}.
\end{equation}
The existence of a Liapounov function is reminiscent of other similar conditions about ergodicity of diffusion processes in the
whole space; see, for example \cite{H},\cite{LM}, \cite{B}, \cite{BG}, \cite{LB}.
%\vspace{0.5cm}

\begin{oss}\label{mathcalgoss}\rm{
We observe that
\begin{equation}\label{mathcalg}
\mathcal{G}_{\bar{x},\bar{p}}(y,q,Y)=-\mathcal{L}_{\bar{x},\bar{p}}(y,q,Y)-|\tau^T(y)q|^2
\end{equation}
where, for any $(\bar{x},\bar{p})\in \R^n \times \R^n,$ $\mathcal{L}_{\bar{x},\bar{p}}$ is the linear operator
$$
\mathcal{L}_{\bar{x},\bar{p}}(y,q,Y)=(b(y)+2\tau(y)\sigma^T(\bar{x},y)\bar{p}) \cdot q  + \tr(\tau \tau^T(y) Y),
$$
which is the infinitesimal generator of the stochastic process
$$
dY_t=(b(Y_t)+2\tau(Y_t)\sigma^T(\bar{x},Y_t)\bar{p})dt+\tau(Y_t)dW_t.
$$
Note that we consider the additional term $-|\tau^Tq|^2$ in \eqref{mathcalg} and this is due  to the logarithmic form of the value function $v_\eps$ defined in \eqref{v-eps}, which is in turn motivated by the applications to large deviations we are interested in.}
\end{oss}
Now we prove  the following lemma.
\begin{lem}
Let (E) hold. Then  for any $(\bar{x},\bar{p})\in \mathbb{R}^n\times\R^n$ there exists a  Liapounov-like function for the operator $\mathcal{G}_{\bar{x},\bar{p}}$.
\end{lem}

%\vspace{0.2cm}

%\vspace{0.2cm}

\begin{proof}
Note that a key role in the following proof is played by the behavior of the drift $b$ at infinity, which  is encoded by  assumption (E). 

We take
\begin{equation}\label{liap}
\chi=a |y|^2,
\end{equation}
and by (E) and the boundedness of $\tau$,  we have for $|y|\geq R$
\begin{equation}\label{T}
-b(y)\cdot D\chi(y)-|\tau^T(y)D\chi(y)|^2\geq  2a B |y|^2 -4 a^2 T|y|^2-2a|b||y|,
\end{equation}
where $T>0$ depends on $||\tau||_\infty$. Then by taking
\begin{equation}\label{a}
a< \frac{B}{2T},
\end{equation} the other terms in $\mathcal{G}$ being negligible because of the boundedness of $\tau$ and $\sigma$, we finally get  \eqref{lyap}. \\
%\begin{oss}\label{}\rm{
\end{proof}
%\begin{oss}\label{liaptheta}\rm{
%Note that analogously we can find
%$\chi_{\bar{x},\bar{p}}:=\chi$ such that $\chi(y) \to + \infty \mbox{ as } |y| \to + \infty$ and if $\mathcal{G}[\chi]:=\mathcal{G}_{\bar{x},\bar{p}}(y,D\chi(y),D^2\chi(y))$ then 
%\begin{equation}\label{lyap}
% \mathcal{G}[\chi]\to -\infty \mbox{ as } |y| \to +\infty \mbox{ in the viscosity sense}.
%\end{equation}
%}
%\end{oss}
%\vspace{0.2cm}
\begin{oss}\label{ossliaperg}\rm{

We observe that condition (E) reminds classical conditions for ergodicity, see for example \cite{BCM}. In particular we recall  the so-called  \textit{recurrence} condition used by  Pardoux and Veretennikov \cite{PV1}, \cite{PV2}, \cite{PV3} namely
\begin{equation}\label{rec}
b(y)\cdot y\to - \infty \quad \mbox{as } |y|\to + \infty.
\end{equation}
Note  that (E) is stronger than \eqref{rec}. The main reason is that in our context we need to have some additional information on the rate of decay of $b\cdot y$,  in particular we need it to be at least quadratic in order to compete with the  quadratic growth (in the gradient  term) of $\mathcal{G}$ (see also  Remark \ref{mathcalgoss}). 
%We note again that  we assume condition (U),  which is stronger than \eqref{liapcond}, for technical reasons  linked to the proof of the Lipschitz bound for the corrector  (we refer also to Remark \ref{remu}).
}
\end{oss}

%\begin{oss}\rm{
%We remark that the results of Section \ref{superunb} and of  subsection \ref{parerg} of Section \ref{critunb} still holds without  \eqref{sigmaipo} and by replacing (U), 1) with  condition \eqref{liapcond}.
%}
%\end{oss}
%\vspace{-0.3cm}

%First we state the result for the critical case $\alpha=2$ in the following lemma.

\section{The critical case: $\alpha=2$}\label{critunb}
\label{sec:casocritico}
\label{3}
%We recall that e
\subsection{Key preliminary results}\label{parergichi}
For any $(\bar{x},\bar{p}) \in \R^n \times \R^n$, the ergodic problem is finding a constant $ \lambda \in \R$ such that the following equation
\begin{equation}\label{eqn:cellacriticounb}
\lambda- \tr(\tau \tau^T(y) D^2w(y))-|\tau^T(y) Dw|^2 -(b(y) +2 (\tau(y) \sigma^T(\bar x,y) \bar{p} ) \cdot Dw(y)- |\sigma(\bar x,y)^T\bar{p}|^2=0. 
\end{equation}
has a viscosity solution $w$. This kind of ergodic problems have been studied by Ichihara  \cite{ichiprimo} and Ichihara and Sheu \cite{ichisheu}. We refer in particular to   Theorem $2.4$ of \cite{ichiprimo}, which we recall in the following proposition. 

Denote \begin{equation}\label{Phi}
\Phi=\{w \in C^2(\mathbb{R}^m) \,:\, \mbox{ there exists } C<0 \mbox{ such that } w(y)\leq C(1+|y|)\}.
\end{equation}
%\vspace{0.2cm}

\begin{prop}\label{resichi}
Let assumption (E) hold. There exists a constant $\lambda^*\in \R$ such that \eqref{eqn:cellacriticounb} admits a classical solution $w \in C^2(\R^m) $ if and only if $\lambda \leq \lambda^*$. Moreover, if $(\lambda,w)$  is a solution of \eqref{eqn:cellacriticounb} and $w \in \Phi$,
then $\lambda=\lambda^*$.
\end{prop}
%\vspace{0.2cm}

\begin{oss}\rm{
 We  remark that  Theorem $2.4$ is proved for  Hamiltonians which are convex in the gradient variable, whereas in our case the Hamiltonian is concave. The two cases are equivalent, since if we have a solution $w$ of \eqref{eqn:cellacriticounb}, then $-w$ is a solution of 
\begin{equation}\label{eqergichi}
-\lambda- \tr(\tau \tau^T(y) D^2w(y))+H(y,Dw(y))=0,
\end{equation}
where
\begin{equation}\label{convex}
H(y,q)=-b(y) \cdot q + |\tau(y)^Tq|^2-2 \tau(y)\sigma(\bar{x},y)^T \bar{p} \cdot q+|\sigma(\bar{x},y)\bar{p}|^2.
\end{equation}
which is now convex in the gradient and satisfies the assumptions of \cite{ichiprimo}. %It is easy to check assumptions $A_1$-$A_3$ and $A_5$. For $A_4'$ we take $\phi_0(y)=-C\sqrt{|y|^2+1}$ and we observe that the leading order term in \eqref{eqergichi} is 
%$$
%-b(y)\cdot D\phi_0(y)= -C\mu\frac{y}{\sqrt{|y|^2+1}} +C\frac{|y|^2}{\sqrt{|y|^2+1}}\to %-\infty \mbox{ as } |y| \to + \infty.
%$$
%where we used assumption (U).
}
\end{oss}

\subsection{The ergodic problem and the effective Hamiltonian}\label{parerg}
%Plugging the
For $\d >0$, we consider the approximate ergodic problem
\begin{equation}\label{cell}
\delta w_\delta +F(\bar{x},y, \bar{p},Dw_\delta, D^2w_\delta)-|\sigma(\bar{x},y)\bar{p}|^2=0,
\end{equation}
where 
\begin{equation}\label{F}
 F(\bar x,y,\bar p, q,Y):= - \tr(\tau \tau^T(y) Y)-|\tau^T(y) q|^2 -b(y)\cdot q -2 (\tau(y) \sigma^T(\bar x,y) \bar{p} ) \cdot q.
  \end{equation}
%and let $w_\delta(\cdot;\bar{x},\bar{p}) \in C^2(\mathbb{R}^m)$ be a  solution of  \eqref{cell}.
%\vspace{0.2cm}

Under our standing assumptions we have the following results.
%\vspace{0.2cm}

\begin{prop}\label{thm:trucell}
Let assumption (E) hold. For any $(\bar{x},\bar{p})$  fixed,  there exists a unique solution $w_\d \in C^2(\R^m)$  of \eqref{cell} satisfying
\begin{equation}\label{boundd}
-\frac{1}{\d}\inf_{y\in\R^m}|\sigma(\bar{x},y)^T\bar{p}|^2\leq w_\d(y)\leq \frac{1}{\d}\sup_{y\in\R^m}|\sigma(\bar{x},y)^T\bar{p}|^2,
\end{equation}
 such that 
%$$
%\sup_{y\in \R^m}|Dw_\d(y;\bar{x},\bar{p})|\leq C_\d
%$$
%for  some constant $C_\d>0$ depending only on  $\d, \bar{x},\bar{p}$ and 
$$
\lim_{\d\to 0} \d w_\d(y)=\mbox{ const }:=\bar{H}(\bar{x},\bar{p})\mbox{ locally uniformly}.
$$
 Moreover $\bar{H}(\bar{x},\bar{p})$ is the unique constant such that \eqref{eqn:cellacriticounb} has a solution $w \in C^2(\R^m)$ satisfying
\begin{equation}\label{loggrowth}
|w(y)|\leq \bar{C}(1+\log(\sqrt{|y|^2+1})) \quad \mbox{ for all } y \in \R^m.
\end{equation}
Finally $w$ is the unique (up to and additive constant)  solution to \eqref{eqn:cellacriticounb} for $\lambda=\bar{H}(\bar{x},\bar{p})$.
%\cap C^{2,\alpha}$ for some $0<\alpha<1$  
\end{prop}
%\vspace{0.2cm}

\begin{oss}\rm{
The  growth estimate \eqref{loggrowth} implies that  $w$ solution of \eqref{eqn:cellacriticounb}  belongs to the class $\Phi$ defined in \eqref{Phi}, allowing us to apply Proposition \ref{resichi}  and deriving the uniqueness of $\bar{H}$.   Note that \eqref{loggrowth} is stronger than the growth required in $\Phi$, in particular it would be enough to prove \eqref{loggrowth} with a linear function of $y$ in the right-hand side.}
\end{oss}
First, we prove the following local gradient bound for the solution of the $\d$-ergodic problem.
%\vspace{0.2cm}
\begin{lem}\label{stimagradunbcrit}
 Let $\delta>0$ and $w_\d \in C^2(\R^m)$ be the unique bounded solution of  \eqref{cell}. 
 Then for all $k\in \mathbb{N}$ and $\bar{x}, \bar{p} \in \mathbb{R}^n$, there exists $C>0$ such that  it holds
\begin{equation}\label{bernunb}
\max_{y\in \bar{B}_k}|D_y w_\delta(y;\bar{x},\bar{p})| \leq C,
\end{equation}
where $B_k$ is the ball with radius $k$ and center $0$ and $C$ depends on $k$ and $\bar{p}$.
\end{lem}
\begin{proof}
We refer to \cite{BCG} where we proved the result by the Bernstein method under the assumption of periodicity; the extension to a local bound follows by cut-off functions arguments, following the derivation of similar estimates in \cite{EI}. We refer  also to \cite{KSsec}, Lemma $2.4$ for an analogous result. We only note that a key role  is played by the coercivity  in the gradient of the ergodic problem, more precisely by the quadratic  term in the gradient $|\tau^TDw_\d|^2.$
\end{proof}

Now we prove Proposition \ref{thm:trucell}.
%\vspace{0.2cm}

\begin{proof}[Proof of Proposition \ref{thm:trucell}]\rm{
We split the proof into two steps.  In step $1$ we prove the existence of a couple $(w,\lambda) \in C(\R^m) \times \R$ solution to \eqref{eqn:cellacriticounb}; in step $2$  we prove that $w\in C^2(\R^m)$, \eqref{loggrowth} and the uniqueness of such $\lambda$. Note that the uniqueness up to an additive constant of $w$ follows from Theorem $2.2$ of \cite{ichiprimo}.  
\begin{step}1-\textit{Existence}
\upshape
We use the methods of \cite{AL}
 based on the small discount approximation \eqref{cell}.
Note that the PDE \eqref{cell} has bounded forcing term $|\sigma^T(\bar{x},y)\bar{p}|^2$ since $\sigma$ is bounded. The existence and uniqueness of a viscosity solution  with the $\d$ dependent bound 
\eqref{boundd}
follows from the Perron-Ishii method and the comparison principle in \cite{DLL}.
Moreover $w_\d\in C^2(\R^m)$, thanks to the Lipschitz uniform estimate of Lemma \ref{stimagradunbcrit} and by elliptic regularity theory of convex uniformly elliptic equations, see \cite{Trudi} and \cite{KSest}.

Now we prove that $\d w_\d(y)$ converges along a subsequence of $\d\to 0$ to the constant $\bar{H}(\bar{x}, \bar{p})$ and $w_\d(y)- w_\d(0)$  converges to the corrector $w$. The hard part is proving equicontinuity estimates for $\d w_\d$. We proceed by a diagonal argument. By the local Lipschitz estimates of Lemma \ref{stimagradunbcrit}, we have
\begin{equation}\label{equilip1}
|w_\delta(y)- w_\d(z)|\leq  C_1|y-z| \quad y, z \in \bar{B}_1,
\end{equation}
where for  convenience we denote  by $C_k$ the constant of Lemma \ref{stimagradunbcrit} in $B_k$ for $k \in \mathbb{N}$. Then $\d w_\d$ is equicontinuous in $\bar{B}_1$. The equiboundedness follows from the comparison principle with constant sub and super solutions, namely $\min_{y\in \R^m} |\sigma(y,\bar{x})^T\bar{p}|^2$ and $\max_{y\in \R^m} |\sigma(y,\bar{x})^T\bar{p}|^2$. Then by Ascoli-Arzela theorem, there exists a subsequence  $\d_n^1w_{\d_n^1}$ of $\d w_\d$, converging  uniformly in $\bar{B}_1$ to a constant $\lambda^1$, since by \eqref{equilip1} we have 
$$
|\d w_\d(y)-\d w_\d(z)|\leq \d C_1|y-z| \quad y,z \in \bar{B}_1
$$
and then
$$
\d w_\d(y)-\d w_\d(z)\to 0 \quad  \forall y,z \in \bar{B}_1 \mbox{ as } \d \to 0.
$$
By the same argument, $\delta_n^1 w_{\d_n^1}$ is equibounded and equicontinous in $\bar{B}_2$. Then, there exists a subsequence  $\d_n^2w_{\d_n^2}$ of $\d^1_nw_{\d^1_n}$, converging uniformly in $\bar{B}_2$ to a constant $\lambda^2$, such that
$$
\lambda^1=\lambda^2=:\lambda.
$$
Similarly, we construct for all $k\in \mathbb{N}$, a sequence  $\{\d_n^kw_{\d_n^k}\}_n$ converging as $n \rightarrow \infty$  uniformly in $\bar{B}_k$ to a constant $\lambda^k=\lambda$. 
%$$
%\left\{
%\begin{matrix}
%\d_1^1w_{\d_1^1} \, \d_1^1w_{\d_1^1} \, \cdots \, \d_n^1w_{\d_n^1}\\
%\d_1^2w_{\d_1^1} \, \d_2^2w_{\d_2^2} \, \cdots \, \d_n^2w_{\d_n^2}\\
%\cdots \cdots \cdots \cdots \cdots \cdots \cdots\\
%\d_1^kw_{\d_1^k} \, \d_2^kw_{\d_2^k} \, \cdots \, \d_n^kw_{\d_n^k}
%\end{matrix}
%\right.\,
%$$
Note that the subsequence $\{\d_n^nw_{\d_n^n}\}_n$ converges locally uniformly to $\lambda$. In fact  for any $k\in \mathbb{N}$ we have that
$
\{\d_n^nw_{\d_n^n}\}_n$ is a subsequence of  $\{\d_n^kw_{\d_n^k}\}_n$ for all $n\geq k,
$
from which we deduce that $\{\d_n^nw_{\d_n^n}\}_n$ converges uniformly in $\bar{B}_k$ for all $k\in \mathbb{N}$. 

Now define $v_\d:=w_\delta(y)-w_\d(0)$. Notice that, for all $k$, $v_\d$ is equibounded in $\bar{B}_k$, since, by Lemma \ref{stimagradunbcrit}, we have
$$
|v_\d(y)|=|w_\delta(y)-w_\d(0)|\leq C_k|y|, \quad y \in \bar{B}_k
$$
and, again by Lemma \ref{stimagradunbcrit}, $v_\d$ is equicontinuous  in $B_k$ since 
$$
|v_\d(y) -v_\d(z)| =|w_\d(y)-w_\d(z)| \leq C_k|y-z|, \quad y,z \in \bar{B}_k.
$$
By an analogous diagonal argument, we find sequences $\{v_{\d_n^k}\}_n$ such that
$
v_{\d_n^{k+1}}$ is a subsequence of $ v_{\d_n^k}$,
and converges  uniformly in $\bar{B}_k$ to a function $v^k$. Moreover for all $k \in \mathbb{N}$, we have 
$$
v^{k+1}(y)=v^k(y) \quad y \in \bar{B}_k.
$$
%The sequence $\{v^n_{\d_n}\}_n $ converges to $v_k$ uniformly in $B(0,k)$ for all $k$. 
Then, if we define $w \, : \, \mathbb{R}^m\to \mathbb{R}$ such that 
\begin{equation}\label{v}
w(y)=v^k(y)  \quad y \in \bar{B}_k, 
\end{equation}  
we conclude that
\begin{equation}\label{convdsucc}
\{v_{\d^n_n}\}_n \rightarrow w \quad \mbox{locally uniformly}.
\end{equation}
Now we prove that $(\lambda, w)$ satisfy \eqref{eqn:cellacriticounb}. From \eqref{cell} we get
\begin{equation}\label{discount1}
\d v_\d +\d w_\d(0)+F(\bar x,y,\bar p, D_yv_\d,D^2_{yy}v_\d)-|\sigma^T(\bar x,y)\bar{p}|^2=0 ,\quad\text{in }\R^m .
\end{equation}
Since $v_\d$ is locally equibounded, $\delta v_\d \rightarrow 0$ locally uniformly and  the claim follows recalling that  $\d w_\d\rightarrow \lambda$ and using the stability property of viscosity solutions. 
 
Finally the corrector inherits the property \eqref{bernunb} of Lemma \ref{stimagradunbcrit}, that is, for all $k \in \mathbb{N}$ and $(bar x, \bar p) \in \mathbb{R}^n \times \mathbb{R}^n$, there exists $C>0$ such that
\begin{equation}
\label{bernw}
\max_{y\in \bar{B}_k}|D_yw(y;\bar{x},\bar{p})|\leq C,
\end{equation}
where $C$ depends on $k$ and $\bar{p}$.
\end{step}
\begin{step}2-\textit{Uniqueness of $\lambda$}
\upshape
The uniqueness is given by Proposition \ref{resichi}, once proved that $w \in \Phi$. The $C^2$ regularity follows from the uniform Lipschitz estimate \eqref{bernw} and the regularity theory of convex uniformly elliptic equations, see  \cite{Trudi} and \cite{KSest}.

Note that, in order to prove that $w \in \Phi$, we  prove the (stronger) growth condition \eqref{loggrowth}.
 We prove the claim for the upper bound, since the proof of the lower bound is analogous. 
 
We take the approximate problem \eqref{cell} and we prove that the function $g=C\log(\sqrt{|y|^2+1})$, for some positive constant $C$ large enough, is a supersolution of \eqref{cell}, that is, we prove
\begin{equation}\label{cellg}
\d g(y)-(b(y)+2\tau(y) \sigma^T(\bar{x},y) \bar{p})\cdot Dg-|\tau^T(y)Dg(y)|^2-\tr(\tau\tau^T(y)D^2g)-|\sigma(\bar{x},y)\bar{p}|^2\geq 0.
\end{equation}
Take $|y|\geq R$ where $R$ is defined in (E). By (E) and the boundedness of  $\sigma$, we have
\begin{multline}\label{est}
\d g(y)-(b(y)+2\tau(y) \sigma^T(\bar{x},y) \bar{p})\cdot Dg-|\tau^T(y)Dg(y)|^2-\tr(\tau\tau^T(y)D^2g)-|\sigma(\bar{x},y)\bar{p}|^2\geq \\2CB\frac{|y|^2}{|y|^2+1}-\frac{KC(1+|\bar{p}|)|y|}{|y|^2+1} -KC^2\frac{|y|^2}{(|y|^2+1)^2}-|\sigma(\bar{x},y)\bar{p}|^2,
\end{multline}
%-KC\frac{1}{|y|^2+1}
where $K$ depends on $B>0$ defined in (E) and on $||\sigma(\bar{x},\cdot)||_\infty$. Then, in order to prove that $g$ is a supersolution of  \eqref{cellg}, we prove that the second term in \eqref{est} is non negative. We factorise $\frac{|y|^2}{|y|^2+1}$ and we prove that
\begin{equation}\label{multline}
2CB-\frac{KC (1+|\bar{p}|)}{|y|}-\frac{KC^2}{|y|^2+1} -\sup_y|\sigma^T\bar{p}|^2\frac{|y|^2+1}{|y|^2}\geq 0.
\end{equation}
%-\frac{KC}{|y|^2}
Note that when $y$ goes to infinity in \eqref{multline} the leading order term  is $2CB-\sup_y|\sigma^T\bar{p}|^2$. Then the claim follows by taking $C$ such that
$2CB=2+\frac{3}{2}\sup_y|\sigma^T\bar{p}|^2$ and $y \in \mathbb{R}^m \setminus \bar{B}_{\bar{R}}$ for some $\bar{R}>R$ such that
$$
\frac{KC(1+|\bar{p}|)}{|y|}+\frac{KC^2}{|y|^2+1}\leq 2, \quad \frac{|y|^2+1}{|y|^2}\leq \frac{3}{2}.
$$
%+\frac{KC}{|y|^2}
 Up to now we proved that the function $C\log(\sqrt{|y|^2+1})$ is a supersolution of \eqref{cell}  in $\mathbb{R}^m \setminus B_{\bar{R}}.$ If $\max_{\bar{B}_{\bar{R}}} w_\d \leq 0$ 
then
$$
w_\d(y) \leq \max_{\bar{B}_{\bar{R}}} w_\delta \leq C\log(\sqrt{|y|^2+1})\quad y \in \partial B_{\bar{R}},
$$
and then by the comparison principle we have
$$
w_\d(y) \leq C\log(\sqrt{|y|^2+1})  \quad y \in \mathbb{R}^m.
$$
Now suppose that $\max_{\bar{B}_{\bar{R}}} w_\d \geq 0$ and notice that in this case $C\log(\sqrt{|y|^2+1})+\max_{\bar{B}_{\bar{R}}} w_\d$ is still a supersolution of \eqref{cell} in $\mathbb{R}^m \setminus B_{\bar{R}}.$ Then, again  by the comparison principle, we get
\begin{equation}\label{stimavd}
w_\d(y) \leq C\log(\sqrt{|y|^2+1}) + \max_{\bar{B}_{\bar{R}}} w_\delta \quad y \in \mathbb{R}^m.
\end{equation}  
Since $w_\d$ satisfies \eqref{stimavd}
$$
v_\d(y) =w_\d(y)-w_\d(0)\leq C\log(\sqrt{|y|^2+1})+\max_{\bar{B}_{\bar{R}}} w_\d(y)  -w_\d(0) \quad y \in \mathbb{R}^m.
$$
We estimate the term $\max_{\bar{B}_{\bar{R}}} w_\d(y)  -w_\d(0)$ by Lemma \ref{stimagradunbcrit} and we get
$$
v_\d(y)\leq C\log(\sqrt{|y|^2+1})+C_{\bar{R}}
$$ 
and thanks to \eqref{convdsucc} we conclude \eqref{loggrowth} by taking $\bar{C}=\max\{C, C_{\bar{R}}\}$.

\end{step}}
\end{proof}

%\subsection*{Representation formulas for $\bar{H}$}
%Once $\bar{H}$ is well defined as in Proposition \ref{thm:trucell}, we can prove the %same representation formulas for $\bar{H}$ proved in  Proposition \ref{represform} %in the periodic case. We omit the proof since it is exactly the same to that of %Proposition \ref{represform}. We just notice that in this case in \eqref{meas} $\mu$ %is the invariant probability measure of the process \eqref{x}.
%PROBLEMA: f BOUNDED? SE SERVE ALLORA NON POSSO CONCLUDERE LA FORMULA CON LA $\mu$ %PERCHÈ HO $\beta$ CHE E ILLIMITATO SE NON SONO SUL TORO.
We recall some properties satisfied by $\bar{H}$. For a proof we refer to \cite{BCG}, Proposition $3.3$.
%\vspace{0.2cm}

\begin{prop}\label{prop:contbarh12unb}
Let assumption (E) hold. \begin{enumerate}\itemsep2pt
\item[(a)] $\bar{H}$  is continuous on $\mathbb{R}^n \times \mathbb{R}^n$;

\item[(b)] the function $p \rightarrow \bar{H}(%\bar
{x}, %\bar
{p})$ is convex;

\item[(c)]  
%for every $C>0$ there exists a modulus $\omega_C$ such that
%\begin{equation}\label{eqn:h1}
%|\bar{H}(x',p) -\bar{H}(x,p)| \leq \omega_C(|x-x'|) \quad \mbox{for all }\, \, x, x' \, \, \mbox{and}\, ,\ |p| \leq C;
%\end{equation}
\begin{equation}\label{eqn:bounds}
\inf_{y\in\R^m} |\sigma^T(\bar x,y)\bar p|^2\leq \bar H(\bar x,\bar p)\\ \leq \sup_{y\in\R^m} |\sigma^T(\bar x,y)\bar p|^2 ;
\end{equation}

\item[(d)]
For all $0<\mu<1$ and $x,z,q,p \in \mathbb{R}^n$, it holds
\begin{equation}\label{semi-homo}
\mu\bar{H}\left(x,\frac{p}{\mu}\right)-\bar{H}(z,q) \geq \frac{1}{\mu-1}\sup_{y\in \mathbb{R}^m}|\sigma^T(x,y)p-\sigma^T(z,y)q|^2 .
\end{equation}
% there exists  $C>0$ independent of $p$ such that, for all $x,\bar{x},p \in \mathbb{R}^n$,
%\begin{equation}\label{hgrowth}
%|\bar{H}(x,p)-\bar{H}(\bar{x},p)|\leq C(1+|p|^2)|x-\bar{x}|;
%\end{equation}

 %\item[(e)]

%if
%\begin{equation}\label{uncor}
%\tau(y)\sigma^T(x,y)=0 \quad \forall x\in \mathbb{R}^n, y \in \mathbb{R}^m,
%\end{equation}
%then, for all $x, \bar x, p, \bar p\in\R^n$, 
%\begin{multline}%equation}
%\label{est_x_H}
%\min_{y\in\R^m}\left(|\sigma^T(x,y)p|^2 - |\sigma^T(\bar x,y)\bar p|^2\right)\leq \bar H(x,p) - \bar H(\bar x,\bar p)\\ \leq \max_{y\in\R^m}\left(|\sigma^T(x,y)p|^2 - |\sigma^T(\bar x,y)\bar p|^2\right).
%\end{multline}
\end{enumerate}
\end{prop}
%\vspace{0.2cm}

Finally we observe that equations like \eqref{eqn:cellacriticounb} have been studied in a non compact setting by Khaise and Sheu in \cite{KSsec}. They prove the existence of a constant $\bar{H}$ such that there is a unique (up to an additive constant) smooth  solution $w$ of \eqref{eqn:cellacriticounb} with prescribed growth. Moreover they provide a representation formula for $\bar{H}$ as the convex conjugate of  a suitable operator over a space of measures.
%\vspace{-0.4cm}
\section{The supercritical case: $\alpha >2$}\label{superunb}
 The ergodic problem is finding, for any $(\bar{x},\bar{p})\in \mathbb{R}^n\times \R^n$ fixed, a unique constant $\lambda \in \R$ such that the following   uniformly elliptic linear equation has a viscosity solution $w$
\begin{equation}\label{ergprobsuper}
\lambda- \tr(\tau \tau^T(y) D^2w(y) ) -b(y)\cdot Dw(y) -|\sigma(\bar{x},y)^T\bar{p}|^2=0. 
\end{equation}
This kind of erogic problems  has been studied in \cite{BCM}, see in particular Proposition $4.2$ and  Theorem $4.3$.
%\vspace{0.2cm}

\begin{prop}\label{propips}
Let assumption (E) hold. For any $(\bar{x},\bar{p})\in \mathbb{R}^n\times \R^n$, there exists a unique invariant probability measure $\mu$ for the process 
\begin{equation}\label{ips}
dY_t=b(Y_t)dt+\sqrt{2}\tau(Y_t)dW_t.
\end{equation}
\end{prop}
%\vspace{0.2cm}

\begin{oss}\rm{
For the details we refer to \cite{BCM}, Proposition $4.2$. We just observe that the proof relies strongly on the existence of a Liapounov function as proved in the paragraph \ref{liapfun} for the infinitesimal generator of the process \eqref{ips}, that is, the operator $\mathcal{L}(y,q,Y)= \tr(\tau\tau^T(y)Y)-b(y)\cdot q$.
}
\end{oss}
%\vspace{0.2cm}

Consider the approximate \textit{$\delta$-ergodic problem} for fixed $(\bar{x},\bar{p},\bar{X})$
\begin{equation}\label{eqn:deltacell}
\delta w_{\delta}(y) - |\sigma(\bar{x}, y)^T \bar{p}|^2 - b(y) \cdot D_y w_{\delta}(y)- \tr(\tau(y) \tau(y)^T D_{yy}^2 w_{\delta}(y))=0\, \, \mbox{in} \, \, \mathbb{R}^m.
\end{equation}
We have the following proposition. 
%Note only that in the place of  Lemma \ref{stimagradunb} we use Proposition \ref{globunifwdsopra} which we prove in Section \ref{sectiongradient}.
%\vspace{0.2cm}

\begin{prop}\label{thm:trucellsopra}
Let assumption (E) and (S1) holds. For any fixed $(\bar{x},\bar{p})$ there exists a unique solution $w_\d \in C^2(\R^m)$ of \eqref{eqn:deltacell} satisfying
\begin{equation}\label{bounddsopracritico}
-\frac{1}{\d}\inf_{y\in\R^m}|\sigma(\bar{x},y)^T\bar{p}|^2\leq w_\d(y)\leq \frac{1}{\d}\sup_{y\in\R^m}|\sigma(\bar{x},y)^T\bar{p}|^2
\end{equation}
such that   
 \begin{equation}\label{formulasuper}
 \lim_{\d\to 0} \d w_\d(y)=\int_{\mathbb{R}^m} \! |\sigma(\bar{x}, y) ^T \bar{p}|^2 \, d\mu(y):=\bar{H}(\bar{x},\bar{p}) \mbox{ locally uniformly },
 \end{equation}
 where $\mu$ is the unique invariant probability measure of the process \eqref{ips}.
Moreover there exists a viscosity solution $w\in C^2(\R^m)$ of \eqref{ergprobsuper} with $\lambda=\bar{H}(\bar{x},\bar{p})$ satisfying \eqref{loggrowth}. %Finally $w$ is  unique up to and additive constant.
\end{prop}
First we prove the following local gradient bound for the solution of the $\d$-ergodic problem.
\begin{lem}\label{stimagradunbsopracrit}
 Let (S1) hold. Let $\delta>0$ and $w_\d \in C^2(\R^m)$ be the unique bounded solution  of \eqref{eqn:deltacell}. 
 Then for all $k\in \mathbb{N}$ and $\bar{x}, \bar{p} \in \mathbb{R}^n$, there exists $C>0$ such that  it holds
\begin{equation}\label{bernunbsopra}
\max_{y\in \bar{B}_k}|D_y w_\delta(y;\bar{x},\bar{p})| \leq C,
\end{equation}
where $B_k$ is the ball of radius $k$ and center $0$ and $C$ depends on $k$ and on $\bar{p}$.
\end{lem}
\begin{proof}
 We suppose that $\bar{p} \neq 0$, otherwise $w_\d=0$ is the unique solution of \eqref{eqn:deltacell}. We observe that, by assumption (S1), if $\d \leq \xi |\bar{p}|^2$, $1$ is a subsolution of \eqref{eqn:deltacell}. Then, for such $\d, w_\d \geq 1$.  Let $\bar{y}$ such that $w_\d(\bar{y})=\min_{\bar{B}_k}w_\d(y)$ and denote $M:=w_\d(\bar{y})-1\geq 0$. Let for $ y \in \bar{B}_k$
\begin{equation}\label{defvd}
v_\d(y)=\log(w_\d(y)- M).
\end{equation}
Then $v_\d$ satisfies on $\bar{B}_k$
\begin{equation}
\label{coercivitysuper}
\d(1+ e^{-v_\d(y)}M)-\tr(\tau \tau^T D^2v_\d(y))-b\cdot D v_\d(y) -|\tau^T Dv_\d(y)|^2-e^{-v_\d(y)}|\sigma^T\bar{p}|^2=0
\end{equation}
and
\begin{equation}\label{segnovd}
v_\d\geq 0 \mbox{ and } v_\d(\bar{y})=0.
\end{equation}
Note that for $y \in \bar{B}_k$, we have $1+ e^{-v_\d(y)}M \geq 0$ and $e^{-v_\d(y)}\leq 1$. Then, by the coercivity of \eqref{coercivitysuper} and analogously to the critical case (see the proof of Lemma \ref{stimagradunbcrit}), we prove that there esists some positive constant $C$, depending on $k$ and $\bar{p}$, such that
\begin{equation}\label{locestvd}
\max_{y\in \bar{B}_k}|D_y v_\delta(y;\bar{x},\bar{p})| \leq C.
\end{equation}
By \eqref{defvd} and \eqref{segnovd}, we have for $y \in \bar{B}_k$
$$ 
D w_\d(y)=Dv_\d(y) e^{v_\d(y)}=Dv_\d(y) e^{v_\d(y)-v_\d(\bar{y})}
$$
and, by \eqref{locestvd}, we finally get \eqref{bernunbsopra}.
\end{proof}
%\vspace{0.2cm}
\begin{proof}[Proof of Proposition \ref{thm:trucellsopra}]
For the identification of $\bar{H}$ and in particular for the proof of \eqref{formulasuper} we refer to \cite{BCM},  Theorem $4.3$. For the existence of the corrector we note that the proof can be carried out analogously as in the critical case and we refer to the proof of Proposition \ref{thm:trucell}.  
\end{proof}
We observe that $\bar{H}$ satisfies the properties $(a),(b),(c),(d)$ of Proposition \ref{prop:contbarh12unb}, which can be proved with similar arguments.
\vspace{0.5cm}

%We have the following representation formula for $\bar{H}$. %
%\vspace{0.2cm}

\section{Gradient bounds}\label{sectiongradient}
In this section, we prove global uniform Lipschitz bounds for the solution of the approximate $\d$-ergodic problems and of the true cell problems.
% both in the critical case and in the supercritical case. This is a key property on which we strongly rely in the proof of the convergence result in Section \ref{convergenceunb}. 
The results are stated in the following propositions. 
%($\alpha=2$) and Proposition \ref{globunifwsopra} ($\alpha >2$).

\begin{prop}\label{globunifwd}
Let assumptions (U) and (S) hold. 
Let $w_\d \in C^2(\R^m)$ be the unique bounded solution of  \eqref{cell} for $\alpha=2$ and of \eqref{eqn:deltacell} for $\alpha>2$. Then for all $x,y \in \mathbb{R}^m$ we have
\begin{equation}\label{globlipeq3d}
|w_\d(y;\bar{x},\bar{p})-w_\d(x;\bar{x},\bar{p})|\leq C|x-y|,
\end{equation}
where  $C$ is a positive constant, depending  on $\bar{x},\bar{p}, ||\tau||_\infty, ||\sigma||_\infty, m$, the Lipschitz constants of $\tau, b, \sigma$ and is independent of $\d$.
\end{prop}

As a straightforward corollary of Proposition \ref{globunifwd}, we get the following global  gradient bound for the correctors. 
%+Note that the same result holds for $\alpha>2$ and we do not repeat the statement.
 %\vspace{0.2cm}
 
\begin{prop}\label{globunifw}
Let assumptions (U) and (S) hold. 
%If $\alpha=2$ let (S) hold and if $\alpha>2$ assume $\sigma$ unformly non-degenerate. 
When $\alpha=2$ let $w \in C^2(\R^m)$ be a solution of  \eqref{eqn:cellacriticounb}  for $\lambda=\bar{H}(\bar{x},\bar{p})$ where  $\bar{H}(\bar{x},\bar{p})$ is defined in Proposition \ref{thm:trucell}; when $\alpha>2$ let $w \in C^2(\R^m)$ be the solution (defined   in Proposition \ref{thm:trucellsopra}) of  \eqref{ergprobsuper}  for $\lambda=\bar{H}(\bar{x},\bar{p})$. Then
\begin{equation}\label{globlipeq3}
\sup_{y\in \R^m}|D_yw(y;\bar{x},\bar{p})|\leq C,
\end{equation}
where  $C$ is a positive constant, depending  on $\bar{x},\bar{p}, ||\tau||_\infty, ||\sigma||_\infty, m$ and the Lipschitz constants of $\tau, b, \sigma$.
\end{prop}

The strategy of the proof consists, roughly speaking, in two steps. 
In Step $1$ we prove an H\"older bound not uniform in $\d$ (see Proposition \ref{unifcont}).  The method is essentially based on the Ishii-Lions method and relies mainly on the uniform ellipticity of the equation. 
In Step $2$ we prove the global uniform gradient bound stated in Proposition \ref{globunifwd}. We remark that the proof is non standard mainly because we do not use any compactness or periodicity of the coefficients, namely  our result  holds in all the space and is independent of $\d$.

The proof of Proposition \ref{unifcont} and Proposition \ref{globunifwd} are carried out only for $\alpha=2$ since the case $\alpha>2$ is analogous and even simpler.

%The proof  is based in the critical on the Berstein method  and mainly relies on the presence of   the coercive term in the gradient $|\tau^TDw_\d|^2$ in the equation. 
Note that, thanks to the uniform local estimate previously proved in Lemma \ref{stimagradunbcrit} for $\alpha=2$ and Lemma \ref{stimagradunbsopracrit} for $\alpha>2$ (see respectively Section \ref{critunb} and \ref{superunb}), the main difficulties  come from the behaviour at infinity, which we treat by the assumptions (U) and (S2).  
 % The proof of Proposition \ref{globunifwd} relies mainly on the H\"older bound of  Proposition \ref{unifcont} and on the use of assumptions (U) and (S).

\begin{step}1-Global H\"older bounds
\upshape

%First we prove Lemma \ref{stimagradunbsopra} (Chapter \ref{driftunb} Section \ref{globunifgradest}, which  we use in the proof of Proposition \ref{globunifwsopra}.
%We prove the following H\"older bound in Proposition \ref{unifcont}.

The proof of Proposition \ref{unifcont} is  based on the Ishii-Lions method which allows us to take profit of the uniform ellipticity. As usual in the  Ishii-Lions method, the estimate that we  prove in \eqref{unifconteq} is not uniform in $\d$. This is the main difference between Proposition \ref{unifcont} and Proposition \ref{globunifwd} and, mainly for this reason, the  proof of Proposition \ref{unifcont} is more standard.

%Note also that in the following proof we use first the local gradient bound of Lemma \ref{stimagradunb} (as we will do in the proof of Proposition \ref{globunifwd}), which  allows us in the critical case to weaken the hypothesis on the process $Y_t$ and consider processes satisfying (U). 
Note that in the following proof we do not need assumption (S), which, on the contrary, is fundamental in the proof of Proposition \ref{globunifwd}. 

%For simplicity of exposition, we split the proof into three steps.
%In particular we do not need  Lemma \ref{liminfxylem} proved in step $2$ in the proof of Proposition \ref{globunifwd}. 

%For the rest of the proof, we follow a similar strategy to that of Proposition \ref{globunifwd}, by  proving first the estimate outside a ball of radius $\bar{R}$ large enough, %$\bar{R}\geq R_1$, where $R_1$ is defined  in assumption (U), 
%where we can use  the assumption (U). Then we combine it with the local estimate proved in  Lemma \ref{stimagradunb}. We refer to the end of the proof where we suitably choose $\bar{R}$ large enough.

 %Note also that in this case we do not need assumption (S) which, on the contrary, we used in the proof of Proposition \ref{globunifwd}. 
%}
%\end{oss}
%\vspace{0.2cm}

\begin{prop}\label{unifcont}
 Let assumptions (U) hold.  Let $w_\d \in C^2(\R^m)$ be the unique bounded solution of  \eqref{cell} for $\alpha=2$ and of \eqref{eqn:deltacell} for $\alpha>2$.
 Then  there exists $C_\d>0$ and $\alpha \in (0,1)$ such that 
\begin{equation}\label{unifconteq}
|w_\d(x;\bar{x},\bar{p})-w_\d(y;\bar{x},\bar{p})|\leq C_\d|x-y|^\alpha \quad \mbox{for all  }x,y 
\in \R^m,
%\R^m\setminus \bar{B}_{\bar{R}}
\end{equation}
where $C_\d$ depends on $\d,\alpha, ||\tau||_\infty, ||\sigma(\bar{x},\cdot)||_\infty, \bar{p}$, the Lipschitz constants of $\tau,b, \sigma$ and $\theta$ of \eqref{unifnondeg}.
\end{prop}
%\vspace{0.2cm}

\begin{proof}\rm{
%We follow the same  strategy used in Proposition \ref{globunifw} and we first we prove the estimate outside a ball of radius $\bar{R>0}$. Then we combine it with the estimate on compact sets proved in  Lemma \ref{stimagradunb}. We refer to the end of the proof where we suitably choose $\bar{R}$ large enough. 
We give the proof for $\alpha=2$ since the case $\alpha>2$ is analogous and even simpler.

Throughout the following proof we denote either by $(a,b)$ or $a \cdot b$ the scalar product for any $a,b \in \R^m$.
For convenience of notation in the following we drop the dependence on $\bar{x},\bar{p}$ by denoting the solution of \eqref{cell} by $w_\d$.

 Let $\d>0$ and $\alpha \in (0,1)$ be fixed and consider the function
\begin{equation}\label{Cdinizio}
w_\d(x)-w_\d(y)-C_\d|x-y|^\alpha,
\end{equation}
for some constant $C_\d> 0$ large enough. Note that $C_\d$  will be choosen suitably at the end of  the proof and  will depend on $\d,\alpha, ||\tau||_\infty, ||\sigma(\bar{x},\cdot)||_\infty, \bar{p}$, the Lipschitz constants of $\tau,b,\sigma$ and $\theta$ of \eqref{unifnondeg}. For clearness of exposition, we keep track only of the dependence on $\d$. 

We suppose that
$$
\sup\{w_\d(x)-w_\d(y)-C_\d|x-y|^\alpha \}=M>0.
$$
Let $R>0$ and consider the function
\begin{equation}\label{Phidev}
\Phi(x,y)=w_\d(x)-w_\d(y)-C_\d|x-y|^\alpha -\psi_R(x)-\psi_R(y),
\end{equation}
where 
\begin{equation}\label{psiR}
\psi_R(z)=\psi\left(\frac{\sqrt{|z|^2+1}}{R}\right)
\end{equation}
and $\psi\in C^2([0,+\infty))$ satisfies
 \begin{equation}\label{psir}
\left\{
 \begin{array}{ll}
\psi(s)=2||w_\d||_\infty+1 \quad &\mbox{if } s\geq 1\\
\psi(0)=0, \, \, \psi\geq 0, \, \, \psi'\geq 0,
                 \end{array}
\right.\,
\end{equation}
where we note that $||w_\d||_\infty$  depends on $\d$ as in \eqref{boundd}.
%Analogously as in Proposition \ref{globunifw}, we prove that $(x,y) \in \left(\R^m\setminus \bar{B}_{\bar{R}}\right)\times  \left(\R^m\setminus \bar{B}_{\bar{R}}\right)$.
We claim  that
\begin{equation}\label{MRM}
M_R=\sup \Phi(x,y)\to M \mbox{ as } R \to + \infty.
\end{equation}
In fact
$$
M_R\leq M \quad \,\mbox{ for any } R>0.
$$
On the other hand
$$
M_R\geq w_\d(x)-w_\d(y)-C_\d|x-y|^\alpha -\psi_R(x) -\psi_R(y)	\, \mbox{ for all } x,y \in \mathbb{R}^m, R>0,
$$
then
$$
\lim_{R \to + \infty} M_R\geq w_\d(x)-w_\d(y)-C_\d|x-y|^\alpha \,\mbox{ for all } x,y \in \mathbb{R}^m
$$
and we conclude
$$
\lim_{R \to + \infty} M_R\geq \sup\{ w_\d(x)-w_\d(y)-C_\d|x-y|^\alpha\}=M.
$$
Then we can suppose for $R$ large enough 
\begin{equation}\label{MRpos}
M_R\geq \frac{M}{2}>0.
\end{equation}
We observe that if $\sqrt{|x|^2+1}\geq R$
$$
\Phi(x,y)\leq -1<0
$$
and the same holds when $\sqrt{|y|^2+1}\geq R$. 
Then, there exists $(x_R,y_R)$ point of maximum of $\Phi$ such that 
\begin{equation}\label{MRdef}
M_R=w_\d(x_R)-w_\d(y_R)-C_\d|x_R-y_R|^\alpha -\psi_R(x_R)-\psi_R(y_R).
\end{equation}
Note that $(x_R,y_R)$ depends also on $\d$ and that  we omit the dependence.
Note also that
\begin{equation}\label{xypos}
|x_R-y_R|>0,
\end{equation}
otherwise by \eqref{MRdef} we have
$$
M_R=-\psi_R(x_R)-\psi_R(y_R)
$$
and we get a contradiction by \eqref{MRpos}  and the definition of $\psi_R$. 

By \eqref{MRM}, \eqref{MRpos} and the definition of $\psi_R$, we also have
$$
C_\d|x_R-y_R|^\alpha\leq 2||w_\d||_\infty:=A_\d.
$$
Then 
\begin{equation}\label{xmenoy}
|x_R-y_R|\leq \left(\frac{A_\d}{C_\d}\right)^{\frac{1}{\alpha}}.
\end{equation}
%Take $C_\d$ large enough so that $|x_R-y_R|\leq 1$.

From now on we omit the dependence on $R$ and we write
$$
(x_R,y_R)=(x,y).
$$
The main result is the following lemma.
\begin{lem}\label{eqprimalimrlemh}
Under the above notations and assumption (U), there exist  positive constants $K, K_1, K_2, K_3,K_4$ such that
\begin{multline*}
0	\leq KC_\d \alpha (\alpha-1)|x-y|^{\alpha-2}+KC_\d\alpha |x-y|^{\alpha+1} + K_1C_\d\alpha|x-y|^\alpha+ K_2\alpha C_\d^2|x-y|^{2\alpha-1}\\+K_2o_R(1)C_\d\alpha|x-y|^{\alpha-1}+ K_3\alpha C_\d|x-y|^{\alpha-1}+K_4|x-y| +o_R(1).
\end{multline*}
where by $o_R(1)$ we mean that $o_R(1)\to 0$ as $R \to + \infty$. Moreover $K, K_1, K_2, K_3, K_4$ depends only on $\bar{p}, ||\sigma||_\infty, ||\tau||_\infty$, the Lipschitz constants of $\tau, b, \sigma$ and $\theta$ of \eqref{unifnondeg}.
\end{lem}
\begin{proof}
Let 
\begin{equation}\label{rx1d}
r_x=D_x\psi_R=2R^{-1}\psi'\left(\frac{\sqrt{|x|^2+1}}{R}\right)x(\sqrt{|x|^2+1})^{-1}
\end{equation}
and 
\begin{equation}\label{ry1d}
 r_y=D_y\psi_R=2R^{-1}\psi'\left(\frac{\sqrt{|y|^2+1}}{R}\right)y(\sqrt{|y|^2+1})^{-1},
 \end{equation}
then  for each $\d$ fixed
\begin{equation}\label{opiccoli0d}
|r_x|, |r_y| \leq o_{R}(1), \quad ||D^2\psi_R||_\infty\leq o_{R}(1),
\end{equation}
where $o_R(1)$ means that $\lim_{R\to +\infty} o_R(1)=0$.

We remark that in the rest of the proof we denote by $o_R(1)$  any function such that $o_R(1) \to 0$ as $R\to + \infty$.
We also denote
\begin{equation}\label{s1}
s=C_\d\alpha|x-y|^{\alpha-2}(x-y).
\end{equation}
Note that the function in \eqref{Phidev} is smooth near $(x,y)$ by \eqref{xypos}. Then, since $w_\d$ is a viscosity solution of \eqref{cell} and since $(x,y)$ is a maximum point of the function in \eqref{Phidev}, we have
\begin{multline}\label{eqwd2h}
0\leq \tr(\tau(x)\tau(x)^TD^2w_\d(x))-\tr(\tau(y)\tau(y)^TD^2w_\d(y))+L(x,y)+G(x,y)+E(x,y) \\+F(x,y)+D(x,y)+ o_R(1),
\end{multline}
where we used \eqref{opiccoli0d}  to estimate the $\psi_R$-terms and we denoted
$$
D(x,y)=\d w_\d(y)-\d w_\d(x);
$$
$$
L(x,y)=s\cdot(b(x)-b(y)) +b(y)\cdot r_y+b(x)\cdot r_x;
$$
$$
G(x,y)=|\tau(x)^T(s-r_x)|^2-|\tau(y)^T(s+r_y)|^2;
$$
$$
E(x,y)=2\tau(x)\sigma(\bar{x},x)^T\bar{p}\cdot(s-r_x)-2\tau(y)\sigma(\bar{x},y)^T\bar{p}\cdot (s+r_y);
$$
$$
F(x,y)=|\sigma^T(\bar{x},x)\bar{p}|^2-|\sigma^T(\bar{x},y)\bar{p}|^2.
$$

First we estimate the second order terms in \eqref{eqwd2h}, by proving the following lemma. 
%\vspace{0.2cm}

\begin{lem}\label{dersec0lem}
Under the above notations, we have
\begin{multline}\label{dersec0}
\tr(\tau(x)\tau(x)^TD^2w_\d(x))-\tr(\tau(y)\tau(y)^TD^2w_\d(y))\leq KC_\d \alpha (\alpha-1)|x-y|^{\alpha-2}\\+KC_\d\alpha |x-y|^{\alpha+1} +o_R(1),
\end{multline}
where $K$ is a positive constant (depending on $\theta$ of \eqref{unifnondeg} and on the Lipschitz constant of $\tau$) and by $o_R(1)$ we mean that $\lim_{R\to + \infty}o_R(1)=0$.
% and where $\tilde{o}_R(1)$ and $o^1(x,y)(1)$ are defined respectively in \eqref{opiccoli0} and in \eqref{tauipo}.
\end{lem}
\begin{proof}
We observe that, for any orthonormal basis $e_i, \, i=1,\cdots m$ of $\R^m$, we can write
\begin{equation}\label{traccia0}
\tr(\tau(x)\tau(x)^TD^2w_\d(x))=\sum_{i=1}^m (\tau(x)\tau(x)^TD^2w_\d(x)e_i,e_i)=\sum_{i=1}^m(D^2w_\d(x)\tau(x) e_i,\tau(x) e_i).
\end{equation}
Denote $\phi(t)=C_\d t^\alpha, f(z)=|z|$. 
By the maximum point property and  the second term of \eqref{opiccoli0d}, we get
\begin{multline}\label{maxpopr}
(D^2w_\d(x)p,p)-(D^2w_\d(y)q,q)\leq \phi'(f(x-y))(D^2f(x-y)(p-q),(p-q))\\+\phi^{''}(f(x-y))(Df(x-y),p-q)^2+o_R(1)
\end{multline}
for any $p,q \in \R^m$.\\
Next we remark that $|Df|^2=1$ and therefore, by differentiating this identity, we have $D^2f Df=0$.
By \eqref{unifnondeg}, we can set
$$
e_1=\frac{\tau(x)^{-1}Df(x-y)}{|\tau(x)^{-1}Df(x-y)|}, \quad \tilde{e}_1=-\frac{\tau(y)^{-1}Df(x-y)}{|\tau(y)^{-1}Df(x-y)|}.
$$
 If $e_1, \tilde{e}_1$ are collinear, the we complete the basis with orthogonal unit vectors $e_i=\tilde{e}_i\in e_1^{\top}, \, 2\leq i\leq m$. Otherwise, in the plane $\mbox{span}\{e_1, \tilde{e}_1\}$, we consider a rotation $\mathcal{R}$ of  angle $\frac{\pi}{2}$ and we define
 $$
 e_2=\mathcal{R}e_1, \quad \tilde{e}_2=-\mathcal{R}\tilde{e}_1.
 $$
Since $\mbox{span}\{e_1, e_2\}^\top=\mbox{span}\{\tilde{e}_1, \tilde{e}_2\}^\top$, we can complete the orthonormal basis with unit vectors $e_i=\tilde{e}_i \in \mbox{span}\{e_1, e_2\}^\top, \, 3\leq i \leq m$.

By \eqref{unifnondeg}, we have
$$
\theta \leq \frac{1}{|\tau(x)^{-1}Df(x-y)|^2}\leq ||\tau||_\infty^2.
$$

 Define 
$$
r_1=\tau(x) e_1 \quad t_1=\tau(y) \tilde{e}_1.
$$
Since $|Df|=1$ and $D^2f Df=0$ and by choosing $p=r_1, q=r_1$ in \eqref{maxpopr}, we get
\begin{multline*}
(D^2w_\d(x)r_1,r_1)-(D^2w_\d(y)t_1,t_1)\leq \phi''(f(x-y))(Df(x-y),r_1-t_1)^2+o_R(1)\\=C_\d\alpha(\alpha-1)|x-y|^{\alpha-2}(Df(x-y),r_1-t_1)^2+ o_R(1).
\end{multline*}
Notice that 
\begin{equation}\label{negeff}
\alpha(\alpha-1)<0.
\end{equation}
By \eqref{unifnondeg},  we have 
$$
(Df(x-y),r_1-t_1)^2=\left(\frac{1}{|\tau(x)^{-1}Df(x-y)|^2}+\frac{1}{|\tau(y)^{-1}Df(x-y)|}\right)^2\geq 4\theta.
$$ 
Then
\begin{equation}\label{dersecphi}
(D^2w_\d(x)r_1,r_1)-(D^2w_\d(y)t_1,t_1)\leq 4\theta C_\d\alpha(\alpha-1)|x-y|^{\alpha-2}+ o_R(1).
\end{equation}
Therefore  in the right hand side we have a very negative term by a double effect, first because we will choose $C_\d$ large but also because, by doing so, $|x-y|$ becomes smaller and smaller and $|x-y|^{\alpha-2}$ larger and larger.

%In particular
%\begin{eqnarray*}
%(D^2\phi(x-y)(r'-t')|(r'-t'))&=&\sum_{i}\left(\sum_{j}\frac{\partial \phi}{\partial x_i x_j} (r'-t')_j\right)(r'-t')_i\\&\leq& C\sum_{i}\left(\sum_{j}\frac{\partial \chi}{\partial x_j}\frac{\partial \chi}{\partial x_i}\right)\frac{\partial \chi}{\partial x_i} =C|D\chi|^2
%\end{eqnarray*}
%where
%$$
%C=\sup\{|\tau(x)+\tau(y)|\}
%$$
Now we choose in \eqref{maxpopr} for all $i\in \{1,\cdots, m-1\}$
$$
p=\tau(x) e_i \quad q=\tau(y) \tilde{e}_i.
$$
Since $\tau$ is Lipschitz, we get 
$$
(D^2w_\d(x)\tau(x) e_i,\tau(x) e_i)-(D^2w_\d(y)\tau(y) \tilde{e}_i,\tau(y) \tilde{e}_i)\leq KC_\d\alpha |x-y|^{\alpha+1} +o_R(1),
$$
where $K$ depends on the Lipschits constant of $\tau$.
Then, by summing the previous equation on $i$  and adding \eqref{dersecphi}, we get
\begin{multline*}
 \sum_{i=1}^m(D^2w_\d(x)\tau(x) e_i,\tau(x) e_i) - \sum_{i=1}^m(D^2w_\d(y)\tau(y) \tilde{e}_i,\tau \tilde{e}_i)\leq KC_\d \alpha (\alpha-1)|x-y|^{\alpha-2}\\+KC_\d\alpha |x-y|^{\alpha+1}+o_R(1),
\end{multline*}
 when by $K$ we denote a constant depending on the Lipschitz constant of $\tau$ and on $\theta$. Then, by  \eqref{traccia0} with $e_i$ defined as above (and $\tilde{e}_i$ for $\tr(\tau(y)\tau(y)^TD^2w_\d(y)))$,
 we finally get  \eqref{dersec0}. 
\end{proof}
Then \eqref{eqwd2h} becomes
\begin{multline}\label{eqwd2hhh}
0	\leq KC_\d \alpha (\alpha-1)|x-y|^{\alpha-2}+KC_\d\alpha |x-y|^{\alpha+1} +L(x,y)+G(x,y)\\+E(x,y) +F(x,y)+D(x,y)+ o_R(1),
\end{multline}
Finally we estimate the left terms $D, L, G, E, F$ in \eqref{eqwd2hhh}.
First note that 
$$
D(x,y)=\d w_\d(y)-\d w_\d(x)\leq 0;
$$
First note that, by  \eqref{s1} and since $b$ is Lipschitz, we have
$$
L(x,y)\leq K_1C_\d\alpha|x-y|^\alpha +b(y)\cdot r_y+b(x)\cdot r_x.
$$
where $K_1$ depends on the Lipschitz constant of $b$.
Note that
$$
b(y)\cdot r_y+b(x)\cdot r_x\leq o_R(1).
$$
Indeed, the previous inequality holds from  the second of \eqref{opiccoli0d} when $x,y$ are uniformly bounded in $R$. Now suppose $|x|\to + \infty$ as $R\to + \infty$ (the argument being similar if $|y|\to + \infty$). By   assumption (U) we have
$$
b(x)\cdot r_x=(b-x)\cdot r_x
$$
and by \eqref{rx1d}, we have
$$
x\cdot r_x=2R^{-1}|x|^2\psi'\left(\frac{\sqrt{|x|^2+1}}{R}\right)(\sqrt{|x|^2+1})^{-1}
$$
and since $\psi'\geq0$ by definition of $\psi_R$ we have
\begin{equation}\label{posrxryd}
x\cdot r_x \geq 0.
\end{equation}
Then by \eqref{posrxryd} and \eqref{opiccoli0d}, we get
$$
(b-x)\cdot r_x\leq o_R(1).
$$
Then  
$$
L(x,y)\leq K_1C_\d\alpha|x-y|^\alpha	+o_R(1).
$$

%\begin{equation}\label{g1}
%G(x,y)\geq|\tau(y)^T s|^2-|\tau(x)^Ts|^2+A_1o_{R}(1),
%\end{equation}
%where $A_1>0$ depends on $\sup_y\tau$.
Now we estimate the $G$-term. By the first of \eqref{opiccoli0d}, \eqref{s1} and since $\tau$ is bounded, we have
$$
G(x,y)\leq |\tau^T(x)s|^2-|\tau^T(y)s|^2+K_2o_R(1)C_\d\alpha|x-y|^{\alpha-1}+o_R(1),
$$
where $K_2$ depends on $||\tau||_\infty$. Note that from now on we denote by $K_2$ a constant depending on $||\tau||_\infty$ and the Lipschitz constant of $\tau$ and which may change from line to line.
Since $\tau$ is bounded by \eqref{s1}, we have
$$
|\tau^T(x)s|+|\tau^T(y)s|\leq K_2C_\d \alpha |x-y|^{\alpha-1}
$$
and since $\tau$ is Lipschitz and by \eqref{s1},  we have
$$
|\tau^T(x)s|-|\tau^T(y)s|\leq K_2C_\d\alpha|x-y|^{\alpha}.
$$
Then we get
$$
|\tau^T(x)s|^2-|\tau^T(y)s|^2\leq K_2\alpha C_\d^2\alpha|x-y|^{2\alpha-1},
$$
and we conclude
\begin{equation}\label{g1}
G(x,y)\leq K_2\alpha C_\d^2|x-y|^{2\alpha-1}+K_2o_R(1)C_\d\alpha|x-y|^{\alpha-1}+o_{R}(1).
\end{equation}
Next we estimate $E$ using the boundedness of $\sigma$ and we get
$$
E(x,y)\leq K_3\alpha C_\d|x-y|^{\alpha-1}+o_{R}(1),
$$
where  $K_3>0$ depends on $\bar{p}, ||\tau||_\infty, ||\sigma||_\infty$. 

Finally, by the Lipschitz continuity and boundedness of $\sigma$, we have
$$
F(x,y) \leq K_4|x-y|,
$$
where
$
K_4$ depends on $||\sigma||_\infty$ and the Lipschitz constant of $\sigma$ and on $\bar{p}$.

 Then, by all the previous estimates, \eqref{eqwd2hhh} becomes
 \begin{multline}\label{eqwd2hhhh}
0	\leq KC_\d \alpha (\alpha-1)|x-y|^{\alpha-2}+KC_\d\alpha |x-y|^{\alpha+1}+K_1C_\d\alpha|x-y|^\alpha+ K_2\alpha C_\d^2|x-y|^{2\alpha-1}\\+K_2o_R(1)C_\d\alpha|x-y|^{\alpha-1}+ K_3\alpha C_\d|x-y|^{\alpha-1}+K_4|x-y| +o_R(1).
\end{multline}
This concludes the proof of Lemma \ref{eqprimalimrlemh}.
\end{proof}

We divide \eqref{eqwd2hhhh} by $C_\d|x-y|^{\alpha-2}$ and we get
\begin{multline}\label{ultholder}
 0\leq  K\alpha(\alpha-1)+K\alpha|x-y|^3+K_1\alpha|x-y|^2+K_2\alpha C_\d|x-y|^{\alpha+1}+K_2o_R(1)\alpha|x-y|\\+K_3\alpha|x-y|+K_4C_\d^{-1}|x-y|^{3-\alpha}+o_R(1)C_\d^{-1}|x-y|^{2-\alpha}.
\end{multline}
Note that by \eqref{xmenoy}, we have
\begin{equation}\label{contr1}
|x-y|\leq A_\d^{\frac{1}{\alpha}}C_\d^{-\frac{1}{\alpha}},
\end{equation}
then 
%C_\d^{-1}|x-y|^{2-\alpha}\leq A_\d^{\frac{2-\alpha}{\alpha}}C_\d^{-\frac{2}{\alpha}}.
$$
C_\d^{-1}|x-y|^{3-\alpha}\leq A_\d^{\frac{3-\alpha}{\alpha}}C_\d^{-\frac{3}{\alpha}};
$$
$$
C_\d^{-1}|x-y|^{2-\alpha}\leq A_\d^{\frac{2-\alpha}{\alpha}}C_\d^{-\frac{2}{\alpha}};
$$
$$
C_\d|x-y|^{\alpha+1}\leq A_\d^{\frac{\alpha+1}{\alpha}}C_\d^{-\frac{1}{\alpha}}.
$$
By all the previous estimates and by taking  $R$ large enough such that $o_R(1)\leq 1$, \eqref{ultholder} becomes
\begin{multline}\label{ultholderfine}
0\leq  K\alpha(\alpha-1)+K\alpha A_\d^{\frac{3}{\alpha}}C_\d^{-\frac{3}{\alpha}}+ K_1\alpha A_\d^{\frac{2}{\alpha}}C_\d^{-\frac{2}{\alpha}} + K_2\alpha A_\d^{\frac{\alpha+1}{\alpha}}C_\d^{-\frac{1}{\alpha}}+K_2o_R(1)\alpha A_\d^{\frac{1}{\alpha}}C_\d^{-\frac{1}{\alpha}} \\ + K_3 \alpha A_\d^{\frac{1}{\alpha}}C_\d^{-\frac{1}{\alpha}} +K_4 A_\d^{\frac{3-\alpha}{\alpha}}C_\d^{-\frac{3}{\alpha}}+ o_R(1)A_\d^{\frac{2-\alpha}{\alpha}}C_\d^{-\frac{2}{\alpha}}.
\end{multline}
%In order to get a contradiction with \eqref{ultholderfine}, we would like to take $C_\d$ large enough.  %The problem is that the position of  $x,y$ depends on the choice of $C_\d$. In order to solve this difficulty,  we proceed as in Proposition \ref{thm:trucell}, step $3$ and we prove Proposition \ref{unifcont}  when $x,y$ satisfy
%\begin{equation}\label{bigenough}
%x, y\in \R^m \setminus \bar{B}_{\bar{R}},
%\end{equation}
 %where $\bar{R}>R_1,$ where $ R_1$ is defined in (U)  and is large enough. 
 
Then, the claim of the proposition follows by taking $C_\d$ in \eqref{Cdinizio}  large enough in order to get a contradiction with \eqref{ultholderfine}. For example we take $C_\d>  \bar{C}_\d$ where $\bar{C}_\d$ satisfies
\begin{multline*}
K\alpha(\alpha-1)+K\alpha A_\d^{\frac{3}{\alpha}}\bar{C}_\d^{-\frac{3}{\alpha}}+K_1\alpha A_\d^{\frac{2}{\alpha}}\bar{C}_\d^{-\frac{2}{\alpha}} + K_2\alpha A_\d^{\frac{\alpha+1}{\alpha}}\bar{C}_\d^{-\frac{1}{\alpha}}+K_2o_R(1)\alpha A_\d^{\frac{1}{\alpha}}\bar{C}_\d^{-\frac{1}{\alpha}} \\ + K_3 \alpha A_\d^{\frac{1}{\alpha}}\bar{C}_\d^{-\frac{1}{\alpha}} +K_4 A_\d^{\frac{3-\alpha}{\alpha}}\bar{C}_\d^{-\frac{3}{\alpha}}+ o_R(1)A_\d^{\frac{2-\alpha}{\alpha}}\bar{C}_\d^{-\frac{2}{\alpha}}<0.
\end{multline*}
Note that  $\bar{C}_\d$ depends  on $K_i, \, i=1,2,3, 4$ and on $\d, \alpha,K$.

%Then, the global estimate \eqref{unifconteq} of Proposition \ref{unifcont} follows by coupling the estimate outside $\bar{B}_{\bar{R}}$ with the local estimate given by Lemma \ref{stimagradunb} applied with $K=\bar{B}_{\bar{R}}. $ }

}
\end{proof}

\end{step}

\begin{step}3-Proof of Proposition \ref{globunifwd}
\upshape
%\label{globunifgradest}
%We prove the following global uniform gradient bound for the solution of the approximate cell problem \eqref{cell}.
%\vspace{0.2cm}

\begin{proof}
%The main difficulty is to control  the  behaviour at infinity of the gradient and this is where  assumptions (S)  play a major role. 

Note that, under the assumption (U), \eqref{cell} reads for $|y|>R_1$
\begin{equation}\label{cellapp}
\delta w_\delta +F(\bar{x},y, \bar{p},Dw_\delta, D^2w_\delta)-|\sigma(\bar{x},y)\bar{p}|^2=0,
\end{equation}
where 
$$
 F(\bar{x},y,\bar{p}, q,Y):= - \tr(\tau \tau^TY)-|\tau^T q|^2 -(b-y, q) -(2 \tau\sigma^T(\bar{x},y) \bar{p}, q).
$$ 
Note also that throughout the following proof we denote either by $(a,b)$ or $a \cdot b$ the scalar product for any $a,b \in \R^m$.

Let $\bar{R}> R_1$ be large enough (which will be chosen suitably at the end of the proof) and take $C_{\bar{R}}$ the constant of Lemma \ref{stimagradunbcrit} for $k=\bar{R}$. 
Then we have for all $x,y \in \bar{B}_{\bar{R}}$
\begin{equation}\label{stimacomplip}
|w_\delta(x; \bar{x},\bar{p})-w_\d(y; \bar{x},\bar{p})| \leq C_{\bar{R}}|x-y|.
\end{equation}
%where we included the dependence on $\bar{p}$ into $C_{\bar{R}}$ for simplicity.
For convenience of notation in the following we drop the dependence on $\bar{x},\bar{p}$ by denoting the solution of \eqref{cellapp} by $w_\d$.

In this first part of the proof we proceed analogously as in the proof of Proposition \ref{unifcont}. The new part of the proof starts from  Lemma \ref{pointmaxout}. We give a sketch and for all the details we refer to the beginning of the proof of Proposition \ref{unifcont}.

 We proceed by contradiction and we suppose that
\begin{equation}\label{ciniziolip}
\sup\{w_\d(x)-w_\d(y)-C|x-y|\}=M>0,
\end{equation}
where $C$ is a positive constant large enough, that is $C>\max\{C_{\bar{R}}, C_{\bar{R}+1}\}$.

Let $R>0$ and consider the function
\begin{equation}\label{Phidevlip}
\Phi(x,y)=w_\d(x)-w_\d(y)-C|x-y| -\psi_R(x)-\psi_R(y),
\end{equation}
where 
\begin{equation}\label{psiRlip}
\psi_R(z)=\psi\left(\frac{\sqrt{|z|^2+1}}{R}\right)
\end{equation}
where $\psi$ is defined in \eqref{psir}.
By standard argument (see also the proof of Proposition \ref{unifcont}), we prove that 
$$
M_R=\sup \Phi(x,y)\to M \mbox{ as } R \to + \infty,
$$
then we can suppose for $R$ large enough
\begin{equation}\label{mrposlip}
M_R\geq \frac{M}{2}>0,
\end{equation}
and by definition of $\psi_R$ we get that, for $R$ large enough, there exist 
$
(x_R,y_R)
$
 such that 
\begin{equation}\label{MRlip}
M_R=w_\d(x_R)-w_\d(y_R)-C|x_R-y_R|-\psi_R(x_R)-\psi_R(y_R).
\end{equation}
Note also that
\begin{equation}\label{xyposlip}
|x_R-y_R|>0.
\end{equation}

We prove the following lemma, whose result is essential in order to use assumption (U) in the rest of the proof. 

\begin{lem}\label{pointmaxout}
Under the above notations, we have that, for $R$ large enough, there exists a point of maximum $(x_R,y_R)$ of the function $\Phi$
 such that $(x_R,y_R) \in \left(\R^m\setminus \bar{B}_{\bar{R}}\right) \times \left(\R^m \setminus \bar{B}_{\bar{R}}\right)$. Moreover
 \begin{equation}\label{liminfxy}
 \liminf_{R\to + \infty}|x_R-y_R|>0.
 \end{equation}
\end{lem}
%$$
%C|x-y|^\alpha\leq \rho_\d(|x-y|)\rho_\d(\frac{2M}{C}^{\frac{1}{\alpha}})
%$$
%where $\rho_\d$ is the modulus of continuity of $w_\d$.\\
%\vspace{0.2cm}

\begin{proof}
Let $(x_R,y_R)$ be a point of maximum of $\Phi$ defined in \eqref{Phidevlip} (see the above arguments for the existence). If  $(x_R, y_R) \in  \left(\R^m\setminus \bar{B}_{\bar{R}}\right)\times   \left(\R^m\setminus \bar{B}_{\bar{R}}\right)$, the claim is proved. 
Otherwise, there are three possible cases (up to subsequences):
\begin{itemize}
\item[(i)]  $(x_R,y_R) \in \bar{B}_{\bar{R}}\times \bar{B}_{\bar{R}}$;
\item[(ii)]$(x_R, y_R) \in \bar{B}_{\bar{R}} \times \left(\R^m\setminus \bar{B}_{\bar{R}}\right)$; 
\item[(iii)] $(x_R, y_R) \in  \left(\R^m\setminus \bar{B}_{\bar{R}}\right)\times \bar{B}_{\bar{R}}$.
\end{itemize}
%\item[(iv)] $(x_R,y_R) \in \left(\R^m\setminus \bar{B}_{\bar{R}}\right) \times \left(\R^m \setminus \bar{B}_{\bar{R}}\right)$.
 Suppose we are in case (i). We apply  the local estimate on $\bar{B}_{\bar{R}}$ \eqref{stimacomplip} and by the choice of $C$ in \eqref{ciniziolip}, we get a contradiction with \eqref{mrposlip}.
  
Now we deal with case (ii) and we observe that case (iii) can be treated analogously.
We prove that there exists  $z_R\in \R^m\setminus \bar{B}_{\bar{R}}$ such that  $(z_R, y_R)$ is still a maximum point of the function $\Phi$. 
Note that we can suppose that    $y_R \in \mathbb{R}^m\setminus \bar{B}_{\bar{R}+1}$. Indeed, if  $y_R \in \bar{B}_{\bar{R}+1}$, we use  the local estimate on $\bar{B}_{\bar{R}+1}$  and by the choice of $C$ in \eqref{ciniziolip}, we get a contradiction with \eqref{mrposlip}.
Let $z_R, z_R'$ be respectively the points where the segment between $x_R$ and $y_R$ intersects the boundary of $B_{\bar{R}+1}$ and of $B_{\bar{R}}$. Note that 
\begin{equation}\label{segment}
|x_R-y_R|= |x_R-z_R|+|z_R-y_R|
\end{equation}
and
\begin{equation}\label{segment2}
|x_R-z_R|= |x_R-z_R'|+1.
\end{equation}
Then, by \eqref{segment}, we have
$$
\max \Phi= \Phi(x_R,y_R)\leq \Phi(z_R,y_R)+w_\d(x_R)-w_\d(z_R)-C|x_R-z_R|-\psi_R(x_R)+\psi_R(z_R),
$$
 and by the local estimate \eqref{stimacomplip} on $\bar{B}_{\bar{R}+1}$ coupled with \eqref{segment2}, we get
\begin{multline*}
\max \Phi\leq \Phi(z_R,y_R)+C_{\bar{R}+1}|x_R-z'_R|+C_{\bar{R}+1}-C|x_R-z'_R|-C-\psi_R(x_R)+\psi_R(z_R).
\end{multline*}
By the choice of $C$ in \eqref{ciniziolip} we get
$$
\max \Phi \leq C_{\bar{R}+1}-C+\Phi(z_R,y_R)-\psi_R(x_R)+\psi_R(z_R)
$$
and, by taking $R$ large enough so that
$
C_{\bar{R}+1}-C-\psi_R(x_R)+\psi_R(z_R)\leq 0,
$
we conclude
$$
\max \Phi \leq \Phi(z_R,y_R).
$$
Then, for $R$ large enough, $(z_R,y_R)\in \left(\R^m\setminus \bar{B}_{\bar{R}} \right)\times  \left(\R^m\setminus \bar{B}_{\bar{R}}\right)$ is a point of maximum of the function $\Phi$. This conclude the proof of the first claim.

Now we prove \eqref{liminfxy}. By contradiction, we suppose that 
$$
\liminf_{R\to + \infty}|x_R-y_R|=0.
$$
By  \eqref{MRlip} and the definition of $\psi_R$, we have
$$
M_R\leq w_\d(x_R)-w_\d(y_R).
$$
Now we use Proposition \ref{unifcont} and by \eqref{unifconteq},
we get
$$
M_R\leq C_\d|x_R-y_R|^\alpha.
$$
Then, since $M_R\to M>0$, we get the following contradiction
$$
0<\liminf_{R \to + \infty}M_R\leq\liminf_{R \to + \infty} C_\d|x_R-y_R|^\alpha=0,
$$ 
concluding the proof.
\end{proof}

From now on we omit the dependence on $R$ and we write
$$
(x_R,y_R)=(x,y).
$$
We prove the following lemma.
\vspace{0.2cm}

\begin{lem}\label{eqprimalimrlem}
Under the above notations and assumptions, there exists two positive constants $K_1, K_2$ such that
\begin{equation}\label{eqprimalimrdue}
 C|x-y|\leq CK_1g(x,y)|x-y|+K_2|x-y|+o_R(1),
\end{equation}
where $g\, :\, \mathbb{R}^m\times \mathbb{R}^m\rightarrow \mathbb{R}^{+}$ is  such that $\forall \eps>0$ there exists $R_\eps$ such that $g(x,y)\leq \eps$ for all $|x|,|y|\geq R_\eps$. Moreover $K_1, K_2$ depends only on $\bar{p}, ||\sigma||_\infty, ||\tau||_\infty$ and by $o_R(1)$ we mean that $\lim_{R\to+\infty}o_R(1)=0$.
\end{lem}
%\vspace{0.2cm}
\begin{oss}\rm{
Note that $C|x-y|$, on the left side in \eqref{eqprimalimrdue},  remains striclty positive for $R\to + \infty$ (by Lemma \eqref{pointmaxout}). This term stems from the Ornstein-Uhlenbeck term $-(b-y) \cdot Dw_\d$ in the ergodic problem \eqref{cellapp}.
}
\end{oss}
\begin{proof}
We denote 
\begin{equation}\label{rxlip}
r_x:=D\psi_R(x)=R^{-1}\psi'\left(\frac{\sqrt{|x|^2+1}}{R}\right)x(\sqrt{|x|^2+1})^{-1}
\end{equation}
\begin{equation}\label{rylip}
 r_y:=D\psi_R(y)=R^{-1}\psi'\left(\frac{\sqrt{|y|^2+1}}{R}\right)y(\sqrt{|y|^2+1})^{-1}.
\end{equation}
We remark that
\begin{equation}\label{opiccoli1lip}
|r_x|, |r_y| \leq R^{-1}||\psi'||_\infty,
\end{equation}
where $||\psi'||_\infty$ depends on $\d$. Similarly we argue for the second derivatives of $\psi_R$ and we get
\begin{equation}\label{opiccoli2lip}
 ||D^2 \psi_R(z)||_\infty \leq o_{R}(1),
\end{equation}
where $o_R(1)$ means that $\lim_{R \to + \infty} o_R(1)=0$.

 Note that in the rest of the proof we denote by $o_R(1)$  any function respectively such that $o_R(1)\to 0$ as $R \to + \infty$. %, with some abuse of notations with \eqref{opiccoli}.\\
%\begin{equation}\label{ishii}
%\begin{pmatrix}X \quad 0 \\0 \quad -Y \end{pmatrix}\leq \begin{pmatrix} D^2\phi(x,y) \quad -D^2\phi(x,y) \\-D^2\phi(x,y) \quad D^2\phi(x,y) \end{pmatrix},
%\end{equation}
We also denote
\begin{equation}\label{slip}
s=C\frac{x-y}{|x-y|}.
\end{equation}
Notice that the function in \eqref{Phidevlip} is smooth since  for $R$ big enough $x\neq y$ by Lemma \ref{pointmaxout}. Then, since $w_\d$ is a viscosity solution of \eqref{cellapp} and since $(x,y)$ is a maximum point of the function in \eqref{Phidevlip}, we have
\begin{multline}\label{eqwd00}
L(x,y)\leq \tr(\tau\tau^TD^2w_\d(x))-\tr(\tau\tau^TD^2w_\d(y))+o_R(1)+G(x,y)\\+E(x,y) +F(x,y)+D(x,y),
\end{multline}
where  we used  \eqref{opiccoli1lip} and \eqref{opiccoli2lip}  to estimate the $\psi_R$-terms and
where  we denote
$$
D(x,y)=\d w_\d(y)-\d w_\d(x);
$$
$$
L(x,y)=(s,(x-y)) -(b-y,r_y)-(b-x,r_x);
$$
$$
G(x,y)=|\tau^T(s+r_x)|^2-|\tau^T(s-r_y)|^2;
$$
$$
E(x,y)=(2\tau\sigma(\bar{x},x)^T\bar{p},s+r_x)-(2\tau\sigma(\bar{x},y)^T\bar{p},s-r_y);
$$
$$
F(x,y)=|\sigma^T(\bar{x},x)\bar{p}|^2-|\sigma^T(\bar{x},y)\bar{p}|^2.
$$
%\vspace{0.5cm}

We estimate each term in \eqref{eqwd00}.
The most important terms is $L$ since it gives rise to the left order term $C|x-y|$ in \eqref{eqprimalimrdue}. Indeed by \eqref{slip}, we have
$$
L(x,y)\geq C|x-y| -(\mu-y)\cdot r_y-(\mu-x)\cdot r_x
$$
and notice that  by  \eqref{rxlip} and  \eqref{rylip} we have
$$
x\cdot r_x=R^{-1}|x|^2\psi'\left(\frac{\sqrt{|x|^2+1}}{R}\right)(\sqrt{|x|^2+1})^{-1}
$$
and
$$
y\cdot r_y=R^{-1}|y|^2\psi'\left(\frac{\sqrt{|y|^2+1}}{R}\right)y(\sqrt{|y|^2+1})^{-1} 
$$
and since $\psi'\geq0$ by definition of $\psi_R$, we have
\begin{equation}\label{posrxry1lip}
x\cdot r_x\geq 0,\quad y\cdot r_y \geq 0.
\end{equation}
By \eqref{posrxry1lip} and \eqref{opiccoli1lip}, we get
$$
-(b-y)\cdot r_y-(b-x)\cdot r_x\geq o_R(1),
$$
and then 
$$
L(x,y)\geq C|x-y|	+o_R(1).
$$
Then by the  previous estimates we get
\begin{equation}\label{eqwd}
C|x-y|\leq \tr(\tau\tau^TD^2w_\d(x))-\tr(\tau\tau^TD^2w_\d(y))+o_R(1)+G(x,y)+E(x,y) +F(x,y)+ D(x,y).
\end{equation}
Now we estimate the remaining terms in the right-hand side of \eqref{eqwd}.
First note that 
$$
D(x,y)=\d w_\d(y)-\d w_\d(x)\leq 0.
$$
By \eqref{opiccoli1lip} and \eqref{slip}, we have
%\begin{equation}\label{g1}
%G(x,y)\geq|\tau(y)^T s|^2-|\tau(x)^Ts|^2+A_1o_{R}(1),
%\end{equation}
%where $A_1>0$ depends on $\sup_y\tau$.
\begin{equation}\label{g1lip}
G(x,y)\leq o_{R}(1).
\end{equation}
Next, by (S) (that is (S2), for $\alpha=2$) and the boundedness of  $\sigma$, we have
$$
E(x,y)\leq CK_1g(x,y)|x-y|+o_{R}(1),
$$
where  $K_1>0$ depends on $\bar{p}, ||\tau||_\infty, ||\sigma||_\infty$.

By the Lipschitz continuity and boundedness of $\sigma$, we have
$$
F(x,y) \leq K_2|x-y|,
$$
where
$
K_2$ depends on $||\sigma||_\infty$ and the Lipschitz constant of $\sigma$ and on $\bar{p}$.

Finally we estimate the second order terms in \eqref{eqwd} as follows 
\begin{equation}\label{derseclip}
\tr(\tau\tau^TD^2w_\d(x))-\tr(\tau\tau^TD^2w_\d(y))\leq o_R(1).
\end{equation}
where by $o_R(1)$ we mean that $\lim_{R\to+\infty}o_R(1)=0$.
%where $o_R(1)$ and $o^1(x,y)(1)$ are functions such that $o_R(1)\to0$  as $R\to + \infty$ and $ o^1(x,y)(1)\to0$ as $|x|,|y| \to +\infty$.
The proof of \eqref{derseclip} is analogous to the proof of \eqref{dersec0}, Lemma \ref{dersec0lem}, Proposition \ref{unifcont} and even simpler. Indeed, we use again the following property: if $e_i, \, i=1,\cdots m$ is an orthonormal basis of $\R^m$ and $A$ is a matrix $m\times m$, we have
$$
\tr(A)=\sum_{i=1}^{m}(Ae_i,e_i),
$$
then for any orthonormal basis $e_i, \, i=1,\cdots m$ of $\R^m$, we can write
\begin{equation}\label{traccialip}
\tr(\tau\tau^TD^2w_\d(x))=\sum_{i=1}^m (\tau\tau^TD^2w_\d(x)e_i,e_i)=\sum_{i=1}^m(D^2w_\d(x)\tau e_i,\tau e_i).
\end{equation}
Denote $ f(z)=|z|$. 
We recall that  the function in \eqref{Phidevlip}  is smooth at $(x,y)=(x_R, y_R)$ for $R$ large enough by Lemma \ref{pointmaxout}. Then, since $x,y$ is a maximum point of the function in \eqref{Phidevlip}  and by  \eqref{opiccoli2lip}, we get
\begin{equation}\label{dersecphi1lip}
(D^2w_\d(x)p,p)-(D^2w_\d(y)q, q)\leq C(D^2f(x-y)(p-q),(p-q))+o_R(1)
\end{equation}
for any $p,q \in \R^m$.
Then, in order to prove the claim, it is enough to   choose  in \eqref{dersecphi1lip} for all $ i\in \{1,\cdots m\}$
$$
p=\tau e_i, \quad q=\tau e_i.
$$
Then we get
$$
(D^2w_\d(x)\tau e_i,\tau e_i)-(D^2w_\d(y)\tau e_i,\tau e_i)\leq o_R(1) \, \,\mbox{ for all } i\in \{1,\cdots m\},
$$
 and by summing the previous equation on $i$, we get
\begin{equation*}
 \sum_{i=1}^m(D^2w_\d(x)\tau e_i,\tau e_i) - \sum_{i=1}^m(D^2w_\d(y)\tau e_i,\tau e_i) \leq o_R(1)
\end{equation*}
from which we conclude  \eqref{derseclip}. By coupling all the previous estimates, we get \eqref{eqprimalimrdue} and we conclude the proof of Lemma \ref{eqprimalimrlem}. 
\end{proof}
Now we conclude the the argument as follows.
We use  assumption (S) and by taking $\bar{R}> R_1$ large enough, we  consider $|x|, |y|$ large enough,
%%  The problem is that the position of  $x,y$ depends on the choice of $C$. In order to solve this difficulty,  
%we prove \eqref{lipwd}  when $x,y$ satisfy
%\begin{equation}\label{bigenough}
%x,y\in \R^m \setminus \bar{B}_{\bar{R}},
%\end{equation}
 such that
 \begin{equation}\label{osigmafinlip}
K_1g(x,y)\leq \frac{1}{2}.
 \end{equation}

Now we send $R\to + \infty$ in \eqref{eqprimalimrdue} and divide by $|x-y|$ thanks to Lemma \ref{pointmaxout}, and we get
\begin{equation}\label{previneqlip}
C\leq \frac{C}{2}+ K_2,
\end{equation}
Then, to get a contradiction with \eqref{previneqlip}, it is enough to take $C$ large enough such that
 \begin{equation}\label{Cfin}
 C> 2K_2.
 \end{equation}
 Note that $C, \bar{R}$ depend respectively only on $K_2, R_1$ and in particular, they are independent on $\d$. 
 
 Then the proof follows by taking  $C$  in \eqref{ciniziolip}, such that $C> \max\{C_{\bar{R}}, C_{\bar{R}+1}, 2K_2\}$, where $\bar{R}>R_1$ is such that \eqref{osigmafinlip} holds.
 % (and we recall that $C_{\bar{R}}=C_{\bar{R}}'(1+|\bar{p}|), C_{\bar{R}+1}=C_{\bar{R}+1}'(1+|\bar{p}|)$, where $C_{\bar{R}}', C_{\bar{R}+1}'$ are the constants of Lemma \ref{stimagradunbcrit} for $k=\bar{R}, \bar{R}+1$ respectively.)

\end{proof}
\end{step}

\section{The comparison principle}
 In this section we provide the comparison principle for  the limit PDE 
\begin{equation}\label{eqn:effPDE21}
 v_t -\bar{H}(x,Dv)=0 \quad \mbox{in} \, \, (0,T) \times \mathbb{R}^n, 
\end{equation} 
where $\bar{H}$ is defined in Proposition \ref{thm:trucell} for $\alpha=2$ and in Proposition \ref{thm:trucellsopra} for $\alpha >2$. 

Note that the comparison principle for the limit problem is a crucial ingredient in the proof of the convergence, which we address in the following section. 
%\vspace{0.2cm}

%In this case, differently from the critical  case, we do not need additional assumptions for the comparison principle to hold.  
\begin{thm}
\label{thm:compbarh211} Let assumption  (U) hold. Let $u\in BUSC([0, T]\times\R^n)$ and $v\in BLSC([0,T]\times \R^n)$ be, respectively, a 
%bounded upper semicontinuous 
subsolution and   a %bounded lower semicontinuous 
supersolution to \eqref{eqn:effPDE21} such that $u(0,x)\leq %h(x)\leq
  v(0,x)$ for all $x\in \R^n$. Then $u(x,t)\leq v(x,t)$ for all $x\in \R^n$ and $0\leq t\leq T$. 
\end{thm}
%\vspace{0.2cm}

\begin{proof}
The proof is exactly the same to \cite{BCG}  Theorem $3.5$, in particular is based on  the  properties $(a)$, $(b)$, $(c)$, $(d)$ of Proposition \ref{prop:contbarh12unb} satisfied by the effective Hamiltonian $\bar{H}$. 
\end{proof}

%\begin{oss}\rm{
%We observe that in \cite{BCM} a method is developped to treat singular perturbation problems for sopra%uniformlyelliptic linear equations as the one considered above. However, their techinques seems not suitable %to treat our situation, neither in the supercritical case. This is due essentially to the fact that  the %aim of \cite{BCM} is mainly to study fast financial models with stochastic volatility without doing a large %deviation  analysis. Therefore, they do not consider a logaritmic payoff as we do and for this reason their %equations are linear. 
%and, moreover, the  scaling of their equation is different for ours (they do not do the change $t\to \eps t$ %and then they do not have $\eps$ in the systems).
%}
%\end{oss}

\section{The convergence result}\label{convergenceunb}
In this section we prove  the convergence  of the $v^\eps$  to the unique solution of the limit problem
\eqref{eqn:effprob11}. Throughout this section, let assumptions (U) and (S) hold.  Let $\alpha \geq 2$. We recall that $v^\eps$ denotes  the unique bounded viscosity solution of 
\begin{equation}\label{eqn:equazioneeps112}
\begin{cases} 
 \partial_t v^\eps -H^\eps \left(x,y,D_x v^\eps, \frac{D_y v^\eps}{\eps^{\alpha-1}}, D^2_{xx} v^\eps, \frac{D^2_{yy} v^\eps}{\eps^{\alpha-1}}, \frac{D^2_{xy} v^\eps}{\eps^{\frac{\alpha-1}{2}}}\right)=0 & \, \, \mbox{in} \, \, [0,T] \times \mathbb{R}^n \times \mathbb{R}^m,\\
v^\eps(0,x,y)=h(x) & \mbox{ in }  \mathbb{R}^n\times \mathbb{R}^m. 
\end{cases}
\end{equation}
where
\begin{eqnarray*} 
H^{\eps}(x,y, p, q, X,Y, Z)&:=&   |\sigma^T p|^2+ b\cdot q+ \tr(\tau \tau^T Y) +  \eps\left(\tr(\sigma\sigma^T X) + \phi \cdot p\right) %+
\\ &+& 2\eps^{\frac{\alpha}{2}-1} (\tau \sigma^T p) \cdot q   +2\eps^{\frac{1}{2}}\tr(\sigma\tau^T Z) + \eps^{\alpha-2}|\tau^T q|^2.
\end{eqnarray*}
We state and prove the convergence result. We will make use of the %technique of the
 relaxed semi-limits which we define as follows. 
%Let $v^\eps$ be an equibounded family of %continuous functions. 
The lower  semi-limit $\underline{v}$ is,
$$
\underline{v}(t,x) := \liminf_{\mbox{\footnotesize{$\epsilon \to 0$}}} \{v^{\epsilon}(t_\eps, x_\eps, y_\eps)\, |\, x_\eps \to x, \, t_\eps \to t, \, y_\eps \mbox{ bounded}\}
$$
and the upper  semi-limit $\bar{v}$ is
$$
\bar{v}(t,x) := \limsup_{\mbox{\footnotesize{$\epsilon \to 0$}}} \{v^{\epsilon}(t_\eps, x_\eps, y_\eps)\, |\, x_\eps \to x, \,t_\eps \to t, \, y_\eps\mbox{ bounded}\}.
$$
Since $h$ is bounded, the family $v^\eps$ is equibounded and we have $\bar{v}\in BUSC([0, T]\times\R^n)$ and $\underline{v}\in BLSC([0, T]\times\R^n)$. Notice that by definition, we have
\begin{equation}\label{vupvlow}
\bar{v}(x,t) \geq \underline{v}(x,t).
\end{equation}
%The standing hypotheses of sections \ref{asssistema} and \ref{log-tran} are assumed in this section.
%
%\vspace{0.2cm}

\begin{thm}\label{conv1unb}
Let assumptions (U) and (S) hold. Recall the effective problem
\begin{equation}\label{eqn:effprob11}
v_{t} - \bar{H}(x, Dv)=0 \, \, \mbox{in} \, \, (0,T) \times \mathbb{R}^n \quad v(0,x)=h(x) \, \, \mbox{on} \, \, \mathbb{R}^n
\end{equation}
 where $\bar{H}$ is  defined by Proposition \ref{thm:trucell} for $\alpha=2$  and Proposition \ref{thm:trucellsopra} for $\alpha>2$.
Then
\begin{enumerate}
\item[a)] the upper limit $\bar{v}$  of $v^{\eps}$  is a subsolution of \eqref{eqn:effprob11};
\item[b)] the lower limit $\underline{v}$ is a supersolution of   \eqref{eqn:effprob11};
\item[c)] $v^\eps$ converges uniformly on the compact subsets of $[0,T) \times \mathbb{R}^n \times \mathbb{R}^m$ to the unique viscosity solution of \eqref{eqn:effprob11}.
\end{enumerate}
\end{thm}
\begin{proof} 
Note that, once $a)$ and $b)$ proved, by the definition of semilimits  %we have
%$%$
%\underline{v} \leq \bar{v} %\quad \mbox{in} \,\,[0,T) \times \mathbb{R}^n;
%$%$
% \,\,in $[0,T) \times \mathbb{R}^n$. %Then we can use t
and by the comparison principle (Theorem \ref{thm:compbarh211}) for the effective equation \eqref{eqn:effprob11}, we get
%$$
%\bar{v}(t,x)\leq \underline{v}(t,x),
%$$
%and then%Theorem \ref{thm:teoconv1}  to conclude that
$$
\bar{v}=\underline{v}=v \quad \mbox{in} \,\,[0,T)  \times \mathbb{R}^n
$$
and then, thanks to the properties of semilimits, we get that $v^\eps$ converges locally uniformly to the %function $\bar{v}=\underline{v}$ which is the 
unique bounded solution of \eqref{eqn:effprob11}. 
%\end{step}
Therefore, the main claims which we have to prove are $a)$ and $b)$.
We prove only $a)$ since the proof of $b)$ is analogous.
Moreover, since the proofs for the critical and supercritical case are similar with some minor (and standard) adaptations, we treat only the case $\alpha=2$.

We take  a smooth function $\psi$, and  without loss of generality we  assume  that $\psi$ is coercive in the variable $x$ and  for all compact $K \subset [0,T] \times \mathbb{R}^n$ there exists a constant $C_K>0$ such that 
\begin{equation}\label{dertunb}
|\partial_t \psi(t,x)| \leq C_K \quad \forall (t,x) \in K.
\end{equation}
Let $(\bar{t},\bar{x})$ be a point of strict maximum  of $\bar{v}(t,x) -\psi(t,x)$. 
Let $\eta >0$ and consider the function
 \begin{equation}\label{Phiconv0}
\Phi(t,x,y)=v^\eps(t,x,y)-\psi(t,x)-\eps(w(y)+\eta \chi(y)),
\end{equation}
where  $w$ is the corrector, solution to the ergodic problem \eqref{eqn:cellacriticounb} for $\lambda=\bar{H}(\bar{x},D_x\psi(\bar{t},\bar{x}))$ and $\chi$ is the Liapounov function, that is 
\begin{equation}
\label{liapfinale}
\chi=a |y|^2, \quad a< \frac{1}{2T},
\end{equation}
for some $T>0$ depending on $||\tau||_\infty$ which we defined in \eqref{T}.

By \eqref{loggrowth} and the definition \eqref{liap} of $\chi$, we have for $\eta$ fixed
$$
w(y)+\eta \chi(y) \rightarrow +\infty \mbox{ as } |y| \to + \infty.
$$
Then, there exists $(t_{\eps,\eta},x_{\eps,\eta},y_{\eps,\eta})\in [0,T]\times \mathbb{R}^n\times  \mathbb{R}^m$ point of maximum of $\Phi$ defined in \eqref{Phiconv0}. We denote
$$
(t_{\eps,\eta}, x_{\eps,\eta}, y_{\eps,\eta})=:(t,x,y).
$$
Since $v_\eps$ is a solution of equation \eqref{eqn:equazioneeps112}, we test it as a subsolution with the function $\psi +\eps(w+\eta\chi)$  and by writing
$$
|\tau(y)^T(Dw(y)+\eta D\chi(y))|^2=|\tau(y)^TDw(y)|^2+\eta^2|D\chi(y)|^2+2\eta(\tau(y)^TDw(y),D\chi(y)),
$$
we get
\begin{multline}\label{equaconv}
\psi_{t}(t,x)- \eps \tr(\sigma\sigma(x,y)^T D^2_{xx}\psi(t,x)) - \eps\phi(x,y) \cdot D_x\psi(t,x)- |\sigma(x,y)^T D_x\psi(t,x)|^2 \\-b(y)\cdot Dw(y)-\tr(\tau(y)\tau(y)^TD^2w(y))-2\tau(y)^T\sigma(x,y)^TD_x\psi(t,x)\cdot Dw(y) -|\tau(y)^TDw(y)|^2\\+\eta G_{\eps,\eta}(x,y)\leq 0,
\end{multline}
where, for convenience of notations, we denote
\begin{multline}\label{G}
G_{\eps,\eta}(x,y)=- b(y)\cdot D\chi(y)-\tr(\tau(y)\tau(y)^TD^2\chi(y))-\eta|\tau(y)^TD\chi(y)|^2\\-2\tau(y)^TD w(y) \cdot D\chi(y) -2\tau(y)\sigma(x,y)^TD_x\psi(t,x) \cdot D\chi(y).
\end{multline}
We recall that the corrector $w$ is solution of the ergodic problem \eqref{eqn:cellacriticounb} for $\lambda=\bar{H}(\bar{x},D_x\psi(\bar{t},\bar{x}))$ (see Proposition \ref{thm:trucell}), that is, $w$ satisfies
\begin{multline}\label{wcorr}
\bar{H}(\bar{x},D\psi(\bar{t},\bar{x}))= b(y)\cdot Dw(y)+\tr(\tau(y) \tau(y)^T D^2 w(y))+|\tau(y)^T D w(y)|^2\\+ 2 (\tau(y) \sigma(\bar{x},y)^T D_x \psi(\bar{t},\bar{x}) ) \cdot Dw(y)
 + |\sigma(\bar{x},y)^T D_x\psi(\bar{t},\bar{x})|^2.
\end{multline}
We use \eqref{wcorr} in  \eqref{equaconv} and we get
\begin{multline}\label{eqabs11}
\psi_{t}(t,x)- \eps \tr(\sigma\sigma(x,y)^T D^2_{xx}\psi(t,x)) - \eps\phi(x,y) \cdot D_x\psi(t,x) +\eta G_{\eps,\eta}(x,y)+F_{\eps}(x,y)\\-\bar{H}(\bar{x},D\psi(\bar{t},\bar{x}))\leq 0,
\end{multline}
where we denote
\begin{multline}\label{feps}
F_{\eps}(x,y)=(-2 \tau(y) \sigma(x,y)^T D_x \psi(t,x) + 2\tau(y) \sigma(\bar{x},y)^T D_x\psi(\bar{t},\bar{x}))  \cdot D w(y)
 \\- |\sigma(x,y)^T D_x\psi(t,x)|^2 
 + |\sigma(\bar{x},y)^T D_x\psi(\bar{t},\bar{x})|^2.
\end{multline}

In the following lemma we prove that $(x,t,y)$ are uniformly bounded in $\eps$ and that $x,t \to \bar{x},\bar{t}$ as $\eps \to 0$. 
Note that we split the proof of the  equiboundedness of $(t,x,y)$ into (i) and (ii) in the following lemma only for convenience of exposition.
%\vspace{0.2cm}

%First we prove the following lemma.
\begin{lem}\label{boundpoints}
 Let $\eta >0$ be fixed. Under the above notations and under the assumptions of Theorem \ref{conv1unb}, we have
 \begin{enumerate} 
\item[(i)]
$(x,t) \mbox{ are uniformly bounded in }\eps;$
\item[(ii)]
$y$ is uniformly bounded in $\eps$;
\item[(iii)]
$ (x,t) \to (\bar{x},\bar{t}) \mbox{ as } \eps \to 0.$
\end{enumerate}
\end{lem}
%\vspace{0.2cm}

%We prove $(i)$.\\
%First we prove  that
%$(x,t)$ are uniformly bounded in $\eps$.\\
We split the proof into three steps; in Step 1 we prove (i), in Step 2 we prove (ii) and in Step 3 we prove (iii).
 \begin{proof}[Proof of Lemma \ref{boundpoints}]\rm{
 \begin{step}1 (Proof of (i))
 \upshape
 For all $x'\in \mathbb{R}^n, y' \in \mathbb{R}^m$ and $t'\in(0,T)$ we have
$$
v^\eps(t,x,y) -\psi(t,x)-\eps(w(y) +\eta \chi(y))\geq v^\eps(t',x',y') -\psi(t',x')-\eps(w(y') +\eta \chi(y')),
$$
that is
$$
\psi(t,x)+\eps(w(y) +\eta \chi(y))\leq 2 \sup_{\eps}||v^\eps||_{\infty} +\sup_\eps \left[\psi(t',x') +\eps(w(y')+\eta \chi(y'))\right]
$$
then
\begin{equation}\label{xbd}
\sup_\eps \left[\psi(t,x) +\eps(w(y)+\eta \chi(y)))\right] < \infty.
\end{equation}
Note that \eqref{xbd} implies
\begin{equation}\label{psicoer}
\sup_\eps \psi(t,x)  < \infty.
\end{equation}
Indeed, \eqref{psicoer} follows immediately from \eqref{xbd}  if $|y|$ is bounded in $\eps$; when  $|y|\to + \infty$ it follows since $\eps(w(y)+\eta \chi(y)))$ is positive  thanks to the definition of $\chi$ and the logarithmic growth of $w$ proved in \eqref{loggrowth}.  Then the uniform boundedness of $x$ and $t$ follows  from \eqref{psicoer} and the coercivity of $\psi$.
\end{step}
\begin{step}2 (Proof of (ii))
\upshape
%Now we prove (ii). %the uniform boundedness of $y$. 
We proceed  by contradiction, supposing $
|y|\to + \infty$ as $\eps \to 0$
and we get a contradiction with the equation \eqref{eqabs11} by applying  Lemma \ref{GF}, whose proof is postponed   at the end of the proof of $a)$. We just observe that it essentially relies on  $(i)$ of Lemma \ref{boundpoints} proved in step $1$, on the quadratic growth of the Liapounov function $\chi$ and on the uniform estimate of the gradient of the corrector $w$ (Proposition \ref{globunifw}).
%\vspace{0.2cm}

\begin{lem}\label{GF}
Let assumptions of Theorem \ref{conv1unb} hold. Let $G_{\eps,\eta}(x,y)$ and $F_\eps(x,y)$ be defined respectively in \eqref{G}  and \eqref{feps} and let  $\eta>0$ be fixed.  Then, if 
\begin{equation}\label{ipocontr}
|y|\to + \infty \mbox{ as } \eps \to 0,
\end{equation}
then we have
\begin{enumerate}
\item[(1)]
$
\lim_{\eps\to 0} G_{\eps,\eta}(x,y)=+\infty.
$
\item[(2)]
$
|\lim_{\eps \to 0} F_{\eps}(x,y) |\leq C',
$
for some constant $C'>0$.
\end{enumerate}
\end{lem}
%\vspace{0.2cm}

Then the uniform boundedness of $y$ follows  by coupling $(1)$ and $(2)$ of Lemma \ref{GF}  with equation \eqref{eqabs11} and observing that   $\phi$ and $\sigma$ are bounded, $t, x$ are uniformly bounded in $\eps$ and the time derivative of $\psi$ is bounded by \eqref{dertunb}. 
%\begin{oss}\rm{
%Note that we prove that $y$ is equibounded on the subsequences such that \eqref{ybar1tx} holds.
%}
%\end{oss}
\end{step}
\begin{step}3 (Proof of (iii))
\upshape
 Note that, by Step $1$ and Step $2$, we can suppose that there exists $(\tilde{t}, \tilde{x}, \tilde{y})$  such that,  up to subsequences
\begin{equation}\label{ybar1}
(t,x,y) \to (\tilde{t}, \tilde{x}, \tilde{y}) \quad \mbox{ as } \eps \to 0.
\end{equation}
Since, for all $ t',x',y'$,
$$
v^\eps(t, x, y)-\psi(t, x) -\eps (w(y) +\eta \chi(y))\geq v^\eps(t',x',y') -\psi(t',x')-\eps (w(y')-\eta \chi(y')),
$$
 using the uniform boundedness of $y$ and the definition of upper semi-limit we get
$$
\bar{v}(\tilde{t}, \tilde{x}) -\psi(\tilde{t}, \tilde{x}) \geq \bar{v}(t',x')-\psi(t',x') \quad \forall t',x'.
$$
Then
$$
\tilde{x}=\bar{x}, \quad \tilde{t}=\bar{t}
$$
and
\begin{equation}\label{bartx1}
t \to \bar{t}, \quad x \to\bar{x} \quad \mbox{ as } \eps \to 0,
\end{equation}
concluding the proof of the lemma.
\end{step}}
\end{proof}
%\vspace{0.2cm}

Now we conclude the proof of Theorem \ref{conv1unb} $a)$. 
%\vspace{0.2cm}

Note that from now on when we do the limit as $\eps \to 0$, we mean the limit along the subsequences such that  \eqref{ybar1} (and then also \eqref{bartx1}) hold.

Note that, by $(iii)$ of Lemma \ref{boundpoints} %\eqref{bartx1}, \eqref{ybar1}
 and by definition of the corrector $w$, we have 
\begin{equation}\label{limfeps}
\lim_{\eps \to 0}F_\eps(x,y)=0,
\end{equation}
where $F_\eps$ is defined in \eqref{feps}.
Then, we let $\eps \to 0$  in \eqref{eqabs11} and use again \eqref{bartx1}, \eqref{ybar1} and \eqref{limfeps} to get
\begin{equation}\label{penult1}
\psi_{t}(\bar{t},\bar{x}) +\eta G_{\eta}(\bar{x},\tilde{y})-\bar{H}(\bar{x},D\psi(\bar{t},\bar{x}))\leq 0.
\end{equation}
where
$$
G_\eta(\bar{x},\tilde{y}):=\lim_{\eps\to0}G_{\eps,\eta}(x,y),
$$
%along the subsequences such that \eqref{ybar1} holds 
where and $G_{\eps,\eta}$ is defined in \eqref{G}.

Note that
\begin{multline*}
G_\eta(\bar{x},\tilde{y})=- b(\tilde{y})\cdot D\chi(\tilde{y})-\tr(\tau(\tilde{y})\tau(\tilde{y})^TD^2\chi(\tilde{y}))-\eta|\tau(\tilde{y})^TD\chi(\tilde{y})|^2\\-2\tau(\tilde{y})^TD w(\tilde{y}) \cdot D\chi(\tilde{y}) -2\tau(\tilde{y})\sigma(\bar{x},\tilde{y})^TD_x\psi(\bar{t},\bar{x}) \cdot D\chi(\tilde{y}).
\end{multline*}
We observe that if
$
\tilde{y} \mbox{ is uniformly bounded in } \eta,
$
we send $\eta \to 0$  and we conclude
\begin{equation}\label{ult1}
\psi_t -\bar{H}(\bar{x},D\psi(\bar{t},\bar{x})) \leq 0.
\end{equation}
Otherwise, if
$$
|\tilde{y}| \to + \infty \mbox{ as } \eta \to 0,
$$
we prove analogously as  in Lemma \ref{GF} $(1)$ that for any $\eta$ small enough
$$
\lim_{\eta \to 0}G_\eta(\bar{x},\tilde{y})=+\infty.
$$
Then we can suppose for $\eta$ small
\begin{equation}\label{pos1}
\eta G_{\eta}(\bar{x},\tilde{y})\geq 0
\end{equation}
and by coupling \eqref{pos1} with \eqref{penult1}, we conclude again \eqref{ult1}.
%\vspace{0.4cm}
\end{proof}
Finally we prove Lemma \ref{GF}.
%\vspace{-0.5cm}
\begin{proof}[Proof of Lemma \ref{GF}]\rm{
First we prove $(1)$. Take $\eta, \eps <1$ and consider $|y|\geq R_1$, where $R_1$ is defined in $(U)$. We analyse $G_{\eps,\eta}$ term by term: 
$$
-b(y)\cdot D_y\chi(y)-|\tau(y)^TD_y\chi(y)|^2\geq 2a|y|^2-2a|b||y|-4 a^2 T |y|^2,
$$
by \eqref{liapfinale} and assumption  $(U)$;
$$
-2\tau(y)\sigma(x,y)^TD_x\psi(t,x)\cdot D_y\chi(y)\geq -2aK|D_x\psi(t,x)||y| -2aK,
$$
where from now on we denote by $K>0$ a constant depending only on $||\tau||_\infty, ||\sigma||_\infty$ which may change from line to line.
Note that $|D_x\psi(t,x)|$ is bounded uniformly in $\eps$ by  Lemma \ref{boundpoints} (i) and the smoothness of $\psi$.
We control the growth of the gradient of $w$ by the global estimate \eqref{globlipeq3} proved in Proposition \ref{globunifw} and we get
$$
-2\tau(y)^TD_y\chi(y) \cdot \tau(y)^TD_yw(y)\geq -4aCK|y|,
$$
where $C$ is defined in \eqref{globlipeq3}.
Then, by coupling all the previous estimates, we get
$$
G_{\eps,\eta}(x,y)\geq   (2a -4 a^2 T )|y|^2-2a|b||y|-4aCK|y|-2aK|D_x\psi(t,x)||y| -2aK.
$$
and by the second of \eqref{liapfinale}, we finally get $(1)$.

In order to prove $(2)$, we use  again \eqref{globlipeq3} of Proposition \ref{globunifw}  to get
$$
 \tau(y) \sigma(\bar{x},y)^T D_x\psi(t,x))  \cdot D_y w(y)\geq -KC|D_x\psi(t,x))|,
$$
where  $C>0$ is defined in \eqref{globlipeq3}. Then we conclude since $(t,x)$ are bounded in $\eps$ by  Lemma \ref{boundpoints} (i) and $\tau, \sigma$ are bounded.}
\end{proof}

\section*{Acknowledgment}
\addcontentsline{toc}{section}{Acknowledgment}
Part of this work was developed while the author was a Ph.D student at the Dept. of Mathematics of the Univ. of Padua. The author wants to express her deep gratitude to M. Bardi and A. Cesaroni for suggesting the problem and for the important help given.  The author wishes also to warmly thank Guy Barles and Olivier Ley for the interesting suggestions which strongly contributed to the improvement of some proofs.

% as well as asymptotic smiles and skews. 

\end{document}